\documentclass[11pt]{amsart}

\usepackage{amsmath,amssymb,amscd,amsthm}

\usepackage[dvips]{graphicx}
\usepackage{epsfig}
\usepackage{graphics}
\usepackage{color}

\usepackage{url,hyperref}

\usepackage[matrix,arrow]{xy}

\newtheorem{dfn}{Definition}[section]
\newtheorem{thm}{Theorem}
\newtheorem{prp}[dfn]{Proposition}
\newtheorem{lem}[dfn]{Lemma}
\newtheorem{cor}[dfn]{Corollary}

\theoremstyle{definition}
\newtheorem{rem}[dfn]{Remark}
\newtheorem{exl}[dfn]{Example}
\newtheorem{prb}{Problem}

\def\R{{\mathbb R}}
\def\Z{{\mathbb Z}}
\def\C{{\mathbb C}}

\def\Sph{{\mathbb S}}
\def\CP{{\mathbb C}{\mathrm P}}
\def\RP{{\mathbb R}{\mathrm P}}

\DeclareMathOperator{\sn}{sn}
\DeclareMathOperator{\cn}{cn}
\DeclareMathOperator{\dn}{dn}

\def\phi{\varphi}
\def\epsilon{\varepsilon}

\def\mod{\mathrm{mod}}

\renewcommand{\Re}{\operatorname{Re}}
\renewcommand{\Im}{\operatorname{Im}}
\newcommand{\Qor}{Q_{or}}
\newcommand{\QorC}{\Qor^\C}

\newcommand{\id}{\mathrm{id}}

\title{Deformation of quadrilaterals and addition on elliptic curves}
\author{Ivan Izmestiev}
\date{\today}
\thanks{Supported by the European Research Council under the European Union's Seventh Framework Programme (FP7/2007-2013)/\allowbreak ERC Grant agreement no.~247029-SDModels}
\address{Institut f\"ur Mathematik \\
Freie Universit\"at Berlin \\
Arnimallee 2 \\
D-14195 Berlin \\
 GERMANY}
\email{izmestiev@math.fu-berlin.de}

\begin{document}
\keywords{Folding of quadrilaterals; porism; elliptic curve; biquadratic equation}
\begin{abstract}
The space of quadrilaterals with fixed side lengths is an elliptic curve. Darboux used this to prove a porism on foldings.

In this article, the space of oriented quadrilaterals is studied on the base of biquadratic equations between their angles. The space of non-oriented quadrilaterals is also an elliptic curve, doubly covered by the previous one, and is described by a biquadratic relation between diagonals. The spaces of non-oriented quadrilaterals with the side lengths $(a_1, a_2, a_3, a_4)$ and $(s-a_1, s-a_2, s-a_3, s-a_4)$ turn out to be isomorphic via identification of two quadrilaterals with the same diagonal lengths.

We prove a periodicity condition for foldings, similar to Cayley's condition for the Poncelet porism.

Some applications to kinematics and geometry are presented.
\end{abstract}

\maketitle

\section{Introduction}

\subsection{The Darboux porism}
Let $ABCD$ be a planar quadrilateral. Denote by $D'$ the image of $D$ under reflection in the line $AC$. We call the transformation
\[
F_D \colon ABCD \mapsto ABCD'
\]
the \emph{$D$-folding}. It acts on the set of all quadrilaterals $ABCD$ with $A \ne C$. Define the $C$-folding in a similar way, see Figure \ref{fig:CDFold}.

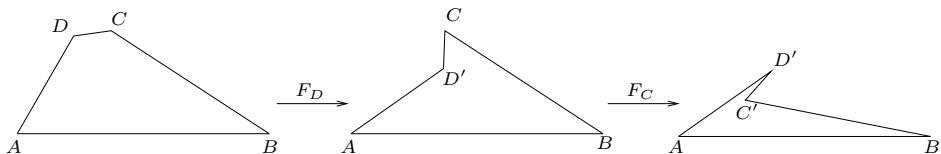
\begin{figure}[htb]
\begin{center}
\begin{picture}(0,0)%
\includegraphics{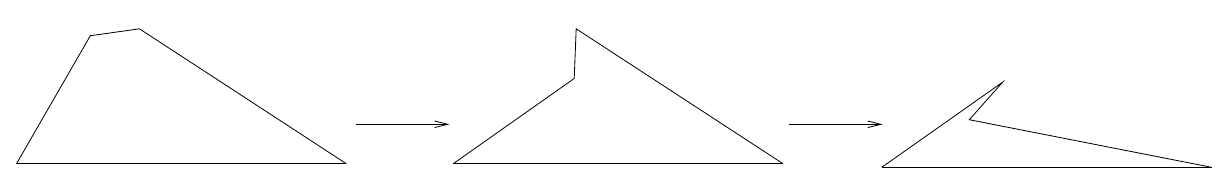}%
\end{picture}%
\setlength{\unitlength}{2072sp}%
\begingroup\makeatletter\ifx\SetFigFont\undefined%
\gdef\SetFigFont#1#2#3#4#5{%
  \reset@font\fontsize{#1}{#2pt}%
  \fontfamily{#3}\fontseries{#4}\fontshape{#5}%
  \selectfont}%
\fi\endgroup%
\begin{picture}(11088,1753)(-1544,199)
\put(7201,659){\makebox(0,0)[lb]{\smash{{\SetFigFont{7}{8.4}{\rmdefault}{\mddefault}{\updefault}{\color[rgb]{0,0,0}$C'$}%
}}}}
\put(7651,1289){\makebox(0,0)[lb]{\smash{{\SetFigFont{7}{8.4}{\rmdefault}{\mddefault}{\updefault}{\color[rgb]{0,0,0}$D'$}%
}}}}
\put(9451,254){\makebox(0,0)[lb]{\smash{{\SetFigFont{7}{8.4}{\rmdefault}{\mddefault}{\updefault}{\color[rgb]{0,0,0}$B$}%
}}}}
\put(6391,254){\makebox(0,0)[lb]{\smash{{\SetFigFont{7}{8.4}{\rmdefault}{\mddefault}{\updefault}{\color[rgb]{0,0,0}$A$}%
}}}}
\put(-269,1784){\makebox(0,0)[lb]{\smash{{\SetFigFont{7}{8.4}{\rmdefault}{\mddefault}{\updefault}{\color[rgb]{0,0,0}$C$}%
}}}}
\put(-989,1694){\makebox(0,0)[lb]{\smash{{\SetFigFont{7}{8.4}{\rmdefault}{\mddefault}{\updefault}{\color[rgb]{0,0,0}$D$}%
}}}}
\put(2476,254){\makebox(0,0)[lb]{\smash{{\SetFigFont{7}{8.4}{\rmdefault}{\mddefault}{\updefault}{\color[rgb]{0,0,0}$A$}%
}}}}
\put(3736,1829){\makebox(0,0)[lb]{\smash{{\SetFigFont{7}{8.4}{\rmdefault}{\mddefault}{\updefault}{\color[rgb]{0,0,0}$C$}%
}}}}
\put(3691,1064){\makebox(0,0)[lb]{\smash{{\SetFigFont{7}{8.4}{\rmdefault}{\mddefault}{\updefault}{\color[rgb]{0,0,0}$D'$}%
}}}}
\put(5536,299){\makebox(0,0)[lb]{\smash{{\SetFigFont{7}{8.4}{\rmdefault}{\mddefault}{\updefault}{\color[rgb]{0,0,0}$B$}%
}}}}
\put(1531,254){\makebox(0,0)[lb]{\smash{{\SetFigFont{7}{8.4}{\rmdefault}{\mddefault}{\updefault}{\color[rgb]{0,0,0}$B$}%
}}}}
\put(-1529,254){\makebox(0,0)[lb]{\smash{{\SetFigFont{7}{8.4}{\rmdefault}{\mddefault}{\updefault}{\color[rgb]{0,0,0}$A$}%
}}}}
\put(1936,929){\makebox(0,0)[lb]{\smash{{\SetFigFont{7}{8.4}{\rmdefault}{\mddefault}{\updefault}{\color[rgb]{0,0,0}$F_D$}%
}}}}
\put(5896,929){\makebox(0,0)[lb]{\smash{{\SetFigFont{7}{8.4}{\rmdefault}{\mddefault}{\updefault}{\color[rgb]{0,0,0}$F_C$}%
}}}}
\end{picture}%
\end{center}
\caption{Composition of $C$- and $D$-foldings.}
\label{fig:CDFold}
\end{figure}

Alternate the $C$- and the $D$-foldings, so that the vertices $A$ and $B$ stay in their places, while $C$ and $D$ jump. If we are lucky, then after several iterations the points $C$ and $D$ return to their initial positions.

\begin{dfn}
A quadrilateral $ABCD$ is called \emph{$n$-periodic}, if it is invariant under the $n$-fold iteration of the composition of $C$- and $D$-foldings:
\[
(F_C \circ F_D)^n (ABCD) = ABCD
\]
\end{dfn}
The quadrilateral from Figure \ref{fig:CDFold} is $3$-periodic, see Figure \ref{fig:Darboux6}. Surprizingly enough, the periodicity depends only on the side lengths.

\begin{figure}[htb]
\begin{center}
\begin{picture}(0,0)%
\includegraphics{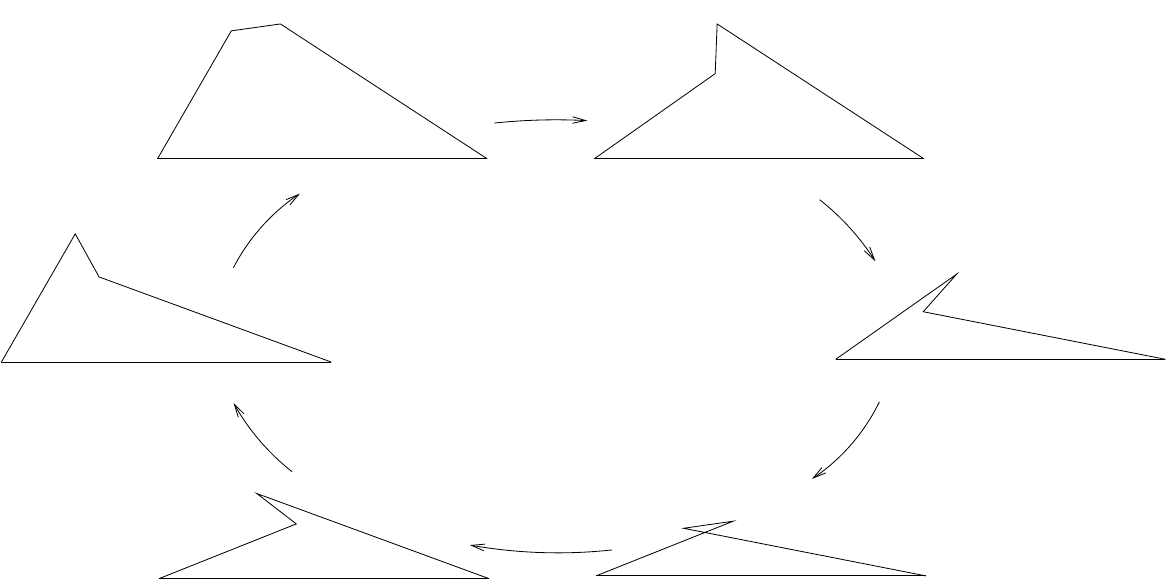}%
\end{picture}%
\setlength{\unitlength}{2072sp}%
\begingroup\makeatletter\ifx\SetFigFont\undefined%
\gdef\SetFigFont#1#2#3#4#5{%
  \reset@font\fontsize{#1}{#2pt}%
  \fontfamily{#3}\fontseries{#4}\fontshape{#5}%
  \selectfont}%
\fi\endgroup%
\begin{picture}(10667,5278)(-2833,-3371)
\put(-269,1784){\makebox(0,0)[lb]{\smash{{\SetFigFont{7}{8.4}{\rmdefault}{\mddefault}{\updefault}{\color[rgb]{0,0,0}$C$}%
}}}}
\put(-989,1694){\makebox(0,0)[lb]{\smash{{\SetFigFont{7}{8.4}{\rmdefault}{\mddefault}{\updefault}{\color[rgb]{0,0,0}$D$}%
}}}}
\put(1531,254){\makebox(0,0)[lb]{\smash{{\SetFigFont{7}{8.4}{\rmdefault}{\mddefault}{\updefault}{\color[rgb]{0,0,0}$B$}%
}}}}
\put(-1529,254){\makebox(0,0)[lb]{\smash{{\SetFigFont{7}{8.4}{\rmdefault}{\mddefault}{\updefault}{\color[rgb]{0,0,0}$A$}%
}}}}
\end{picture}%
\end{center}
\caption{A $3$-periodic quadrilateral. Side lengths are $1, 3, 3\sqrt{5}, 5$.}
\label{fig:Darboux6}
\end{figure}
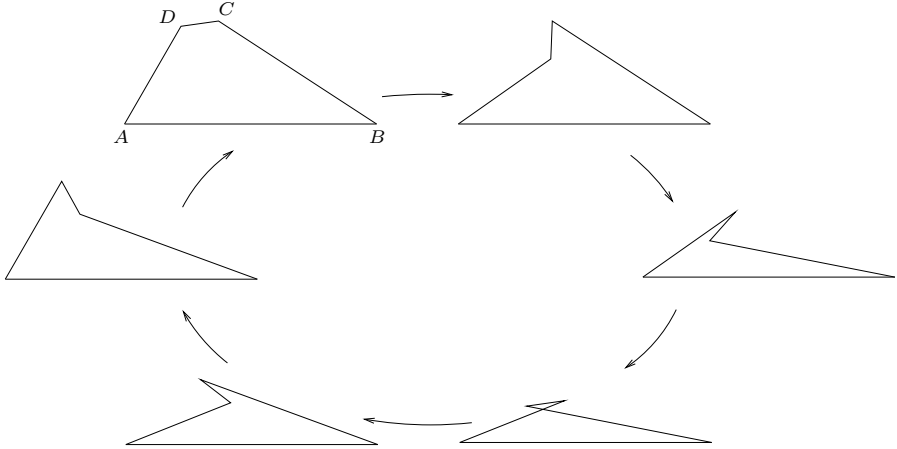


\begin{thm}[Darboux \cite{Dar79}]
\label{thm:Darboux}
If a quadrilateral is $n$-periodic, then every quadrilateral with the same side lengths is $n$-periodic.
\end{thm}

The clue to the Darboux porism is the fact that the space of congruence classes of quadrilaterals with fixed side lengths is an elliptic curve, and the foldings $F_C$ and $F_D$ act on it as involutions. Hence $F_C \circ F_D$ is a translation, just as well as $(F_C \circ F_D)^n$. If a translation has a fixed point (one quadrilateral is $n$-periodic), then it is an identity (all quadrilaterals with the same side lengths are periodic).

Unlike the Poncelet porism \cite{DR14}, its relative, the Darboux porism is much less known and was rediscovered several times in the recent decades.

In the present article we take a closer look at the configuration space $\QorC(a)$ of quadrilaterals with side lengths $a \in \R^4$.
We study also the space $Q^{\C}(a)$ of non-oriented quadrilaterals, which is an elliptic curve doubly covered by $\QorC(a)$. There are some surprising algebraic identities that result in a natural isomorphism between the spaces $Q^{\C}(a)$ and $Q^{\C}(\bar a)$, where $a \mapsto \bar a$ is a certain involution on $\R^4$. Geometrically this leads to ``conjugate pairs'' of quadrilaterals. Finally, we express $n$-periodicity of a quadrilateral in terms of its side lengths.

%
%
%
%
%
%
%

\subsection{Euler-Chasles correspondence and Jacobi elliptic functions}
Introduce the variables
\[
z_1 = \tan\frac{\phi_1}2, \quad z_2 = \tan\frac{\phi_2}2
\]
where $\phi_1$ and $\phi_2$ are adjacent angles of a quadrilateral. Then the space of congruence classes of quadrilaterals, respecting the orientation, with fixed side lengths becomes identified with the algebraic curve
\begin{equation}
\label{eqn:Curve}
c_{22} z_1^2 z_2^2 + c_{20} z_1^2 + c_{02} z_2^2 + 2 c_{11} z_1 z_2 + c_{00} = 0
\end{equation}
Here coefficients $c_{ij}$ depend on the side lengths. This is a special biquadratic equation in two variables. Biquadratic equations are also known under the name of Euler-Chasles correspondences \cite{BV96}. In the past decades they attracted a lot of attention: as a basis for $QRT$-maps in the theory of discrete integrable systems \cite{Dui10}; as a solution of the Yang-Baxter equation \cite{Bax89, Kri81}; as an approach to flexible polyhedra \cite{Sta10, Gai13, Izm_Koko}. 

It is known that in a generic case the curve \eqref{eqn:Curve} can be parametrized as
\[
z_1 = p_1 \phi(t), \quad z_2 = p_2 \phi(t+\tau)
\]
where $\phi$ is a second-order elliptic function with simple poles and zeros. Since any pair of adjacent angles is related by a biquadratic equation, all angles turn out to be scaled shifts of the same elliptic function.

If $a = (a_1, a_2, a_3, a_4)$ are the side lengths of a quadrilateral, then the non-degeneracy condition ensuring that the configuration space $\Qor(a)$ is an elliptic curve is
\[
a_1 \pm a_2 \pm a_3 \pm a_4 \ne 0
\]
for all choices of signs. Put differently, this means $a_{\min} + a_{\max} \ne s$, where $a_{\min}$ and $a_{\max}$ are the lengths of the shortest and the longest side, and $s = \frac{a_1+a_2+a_3+a_4}2$ is the half-perimeter.

In kinematics, the inequality $a_{\min} + a_{\max} < s$ is knows as the Grashof condition. It means that the shortest side can make a full turn with respect to each of its neighbors; also it means that the real part of the configuration space $\Qor(a)$ has two components, that is the quadrilateral cannot be deformed into its mirror image. The following theorem describes a parametrization of $\QorC(a)$ and underlines the difference between Grashof and non-Grashof quadrilaterals.

\begin{thm}
\label{thm:Param}
If $a_1 \pm a_2 \pm a_3 \pm a_4 \ne 0$, then the complexified configuration space $\QorC(a)$ of quadrilaterals with side lengths $a$ is an elliptic curve $\C/\Lambda$.
\begin{enumerate}
\item
If $a_{\min} + a_{\max} < s$, then the lattice $\Lambda$ is rectangular, and the cotangents of the halves of the exterior angles can be parametrized as
\begin{align*}
\cot\frac{\phi_1}2 &= q_1 \dn(t+t_1; k), &\cot\frac{\phi_2}2 &= q_2 \dn(t+t_2; k),\\
\cot\frac{\phi_3}2 &= q_3 \dn(t+t_3; k), &\cot\frac{\phi_4}2 &= q_4 \dn(t+t_4; k)
\end{align*}
Here $k = \sqrt{\frac{a_1a_2a_3a_4}{\bar a_1 \bar a_2 \bar a_3 \bar a_4}}$, the amplitudes $q_i$ are real or purely imaginary, and the shifts $t_i$ satisfy
\[
\Im t_i \in \{0, 2K'\}, \quad t_1 - t_3, t_2 - t_4 \in \{\pm K\}
\]
The real part of the configuration space consists of two components corresponding to $\Im t \in \{0, 2K'\}$.
\item
If $a_{\min} + a_{\max} > s$, then the lattice $\Lambda$ is rhombic, and the cotangents of the halves of the exterior angles can be parametrized as
\begin{align*}
\cot\frac{\phi_1}2 &= q_1 \cn(t+t_1; k), &\cot\frac{\phi_2}2 &= q_2 \cn(t+t_2; k),\\
\cot\frac{\phi_3}2 &= q_3 \cn(t+t_3; k), &\cot\frac{\phi_4}2 &= q_4 \cn(t+t_4; k)
\end{align*}
Here $k = \sqrt{\frac{\bar a_1 \bar a_2 \bar a_3 \bar a_4}{a_1a_2a_3a_4}}$, each of the amplitudes $q_i$ is real or purely imaginary, and the shifts $t_i$ satisfy
\[
\Im t_i \in \{0, 2K'\}, \quad t_1 - t_3, t_2 - t_4 \in \{\pm K\}
\]
The real part of the configuration space has one component corresponding to $t \in \R$.
\end{enumerate}
\end{thm}

The proof of Theorem \ref{thm:Param} is contained in Sections \ref{sec:ParamThm} and \ref{sec:EllReal}. There one can also find exact values of the amplitudes and shifts.
The numbers $\bar a_i$ appearing in the formula for the Jacobi modulus are defined as follows.
\begin{equation}
\label{eqn:BarA}
\bar a_1 = \frac{-a_1 + a_2 + a_3 + a_4}2 = s-a_1, \quad \bar a_i = s - a_i, \ i = 2, 3, 4
\end{equation}

The pair of opposite angles $\phi_1$, $\phi_3$ in a quadrilateral is also subject to a relation of the form \eqref{eqn:Curve}, but with a vanishing coefficient at $z_1z_3$. A scaling of variables brings the equation into the form
\begin{equation}
\label{eqn:Opp}
u^2 + v^2 = 1 + mu^2v^2
\end{equation}
with a real $m<1$.
This curve can be parametrized as
\[
u = f(t;k), \quad v = f(t+K;k)
\]
where $f = \dn$ if $m>0$ and $f = \cn$ if $m<0$. Equation \eqref{eqn:Opp} is closely related to the equation
\[
u^2 + v^2 = a^2 + a^2u^2v^2
\]
suggested by Edwards \cite{Edw07} as a new normal form for elliptic curves. Edwards constructs a parametrization of $u$ and $v$ ``from the scratch'' using a version of theta-functions.

\subsection{The space of non-oriented quadrilaterals}
\label{sec:Conj}
Let $Q^{\C}(a)$ be the space of congruence classes, disregarding the orientation, of quadrilaterals with side lengths $a$. Clearly, the map $\QorC(a) \to Q^{\C}(a)$ is a double cover, easy to describe in terms of the holomorphic parameter of Theorem \ref{thm:Param}. There is an unexpected natural isomorphism between the spaces $Q^{\C}(a)$ and $Q^{\C}(\bar a)$ with $\bar a$ as in \eqref{eqn:BarA}: for every quadrilateral with the side lengths $a$ there is a quadrilateral with the same diagonal lengths and the side lengths $\bar a$. 

\begin{thm}
\label{thm:NonOr}
If $a_1 \pm a_2 \pm a_3 \pm a_4 \ne 0$, then the complexified configuration space $Q^{\C}(a)$ of non-oriented quadrilaterals with the side lengths $a$ is an elliptic curve $\C/\Lambda$ with a rectangular lattice $\Lambda$.

There is a natural isomorphism
\[
Q^{\C}(a) \to Q^{\C}(\bar a)
\]
that identifies two quadrilaterals with the same diagonal lengths.

If $\QorC(a)$ has a rectangular lattice, then $\QorC(\bar a)$ has a rhombic lattice. In particular, the above isomorphism does not lift to the spaces of oriented quadrilaterals. The double covers
\[
\QorC(a) \to Q^{\C}(a) \cong Q^{\C}(\bar a) \leftarrow \QorC(\bar a)
\]
look as shown on Figure \ref{fig:Covers}.
\end{thm}
Theorem \ref{thm:NonOr} is proved in Section \ref{sec:NonOr}.

\begin{figure}[ht]
\centering
\includegraphics{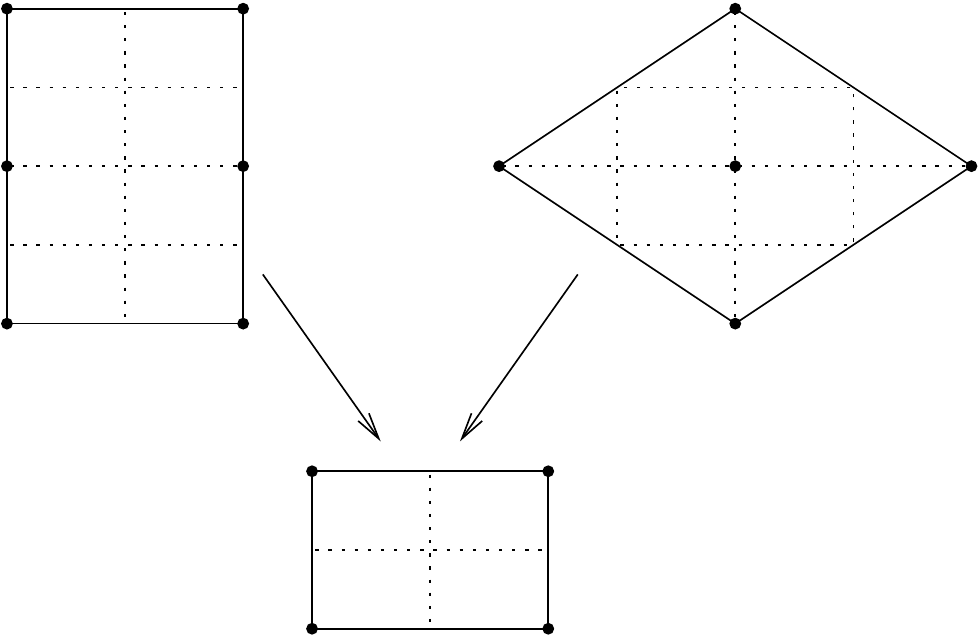}
\caption{Spaces $\QorC(a)$ and $\QorC(\bar a)$ covering $Q^{\C}(a) \cong Q^{\C}(\bar a)$.}
\label{fig:Covers}
\end{figure}

The isomorphism $Q^{\C}(a) \cong Q^{\C}(\bar a)$ turns out to be equivalent to the Ivory theorem, see Figure \ref{fig:Ivory}. Theorem \ref{thm:NonOr} holds also for spherical and hyperbolic quadrilaterals. This extends the Ivory theorem to the sphere and the hyperbolic plane, with ellipses and hyperbolas defined geodesically, see Theorem \ref{thm:IvorySph}.

Computations that allow to express in a particularly nice way the Jacobi modulus, amplitudes, and shifts in Theorem \ref{thm:Param}, and those leading to Theorem \ref{thm:NonOr} are based on a number of identities for dual quadruples $(a,b,c,d)$ and $(\bar a, \bar b, \bar c, \bar d)$. These are collected below.
\begin{gather*}
ab - \bar c \bar d = (s-a-c)(s-b-c), \quad ab - \bar a \bar b = s(s-c-d)\\
abcd - \bar a \bar b \bar c \bar d = s(s-a-b)(s-b-c)(s-a-c)\\
ab + cd = \bar a \bar b + \bar c \bar d, \quad
a^2 + b^2 + c^2 + d^2 = \bar a^2 + \bar b^2 + \bar c^2 + \bar d^2\\
ab - cd = \frac12 (\bar c^2 + \bar d^2 - \bar a^2 - \bar b^2)
\end{gather*}

\subsection{The periodicity condition}
Benoit and Hulin \cite{BH04} studied the periodicity condition by constructing a pair of circles whose Poncelet dynamics is equivalent to the folding dynamics of the quadrilateral.

The following theorem deals with the periodicity disregarding the orientation. Proposition \ref{prp:Period} describes how the period lengths on the curves $Q^{\C}(a)$ and $\QorC(a)$ are related.

\begin{thm}
\label{thm:Period}
A quadrilateral with the side lengths $a$ is $n$-periodic disregarding the orientation if and only if the following condition is satisfied.
\begin{align*}
&\begin{vmatrix}
A_2 & A_3 & \ldots & A_{k+1}\\
A_3 & A_4 & \ldots & A_{k+2}\\
\vdots & \vdots & \ddots & \vdots\\
A_{k+1} & A_{k+2} & \ldots & A_{2k}
\end{vmatrix}
= 0, \quad \text{ if } n = 2k+1\\
&\begin{vmatrix}
A_3 & A_4 & \ldots & A_{k+1}\\
A_4 & A_5 & \ldots & A_{k+2}\\
\vdots & \vdots & \ddots & \vdots\\
A_{k+1} & A_{k+2} & \ldots & A_{2k-1}
\end{vmatrix}
= 0, \quad \text{ if } n = 2k
\end{align*}
Here $A_i$ are the coefficients of the expansion
\[
\sqrt{(x - \Delta(a)^2)(x - \delta(a)^2)(x - \delta(\bar a)^2)} = \sum_{i=0}^\infty A_i x^i
\]
where
\[
\Delta(a) = a_1a_3 + a_2a_4, \quad \delta(a) = a_1a_3 - a_2a_4
\]
\end{thm}

Note that $\Delta(\bar a) = \Delta(a)$ due to the identities from Section \ref{sec:Conj}, so that the quadrilaterals with the side lengths $\bar a$ have the same folding period (disregarding the orientation) as the quadrilaterals with the side lengths $a$. This is also obvious from the isomorphism $Q^{\C}(a) \cong Q^{\C}(\bar a)$ via equal diagonal lengths.

\subsection{Acknowledgments}
Parts of this work were done during the author's visits to IHP Paris and to the Penn State University. The author thanks both institutions for hospitality. Also he wishes to thank Arseniy Akopyan, Udo Hertrich-Jeromin, Boris Springborn, and Yuri Suris for useful discussions.

%
%
%
%
%
%

\section{The space of oriented quadrilaterals in terms of their angles}
\label{sec:Eq}
\subsection{Notation}
A \emph{planar quadrilateral} is for us an ordered quadruple of points $(A, B, C, D)$ in the euclidean plane such that
\[
A \ne B,\quad B \ne C,\quad C \ne D,\quad D \ne A
\]
Two quadrilaterals $ABCD$ and $A'B'C'D'$ are called \emph{directly congruent} if there exists an orientation-preserving isometry $f \colon \R^2 \to \R^2$ such that
\[
f(A) = A',\quad f(B) = B',\quad f(C) = C',\quad f(D) = D'
\]
In this article, we study the set of (direct) congruence classes of quadrilaterals with fixed side lengths. A mechanical interpretation of this is the configuration space of a four-bar linkage with the positions of two adjacent joints fixed and whose bars are allowed to cross each other.

For a quadruple $a := (a_1, a_2, a_3, a_4)$ of real numbers there exists a
quadrilateral with side lengths $a$ if and only if the following inequalities hold:
\begin{subequations}
\begin{equation}
\label{eqn:QuadIneqA}
a_i > 0 \quad \forall i
\end{equation}
\begin{equation}
\label{eqn:QuadIneqB}
a_i < a_j + a_k + a_l \quad \forall i,
\end{equation}
\end{subequations}
where $\{j,k,l\} = \{1,2,3,4\} \setminus \{i\}$ in the second line.

\begin{dfn}
For $a \in \R^4$ satisfying the conditions \eqref{eqn:QuadIneqA} and \eqref{eqn:QuadIneqB}, denote by $\Qor(a)$ the set of direct congruence classes of quadrilaterals $ABCD$ with
\[
AB = a_1, \quad BC = a_2, \quad CD = a_3, \quad DA = a_4
\]
\end{dfn}

\begin{rem}
We require a congruence to preserve the marking of the vertices (or, equivalently, the marking of the sides). This does matter only if the sequence $a$ is symmetric under the action of an element of the dihedral group on $(1, 2, 3, 4)$.
\end{rem}

Denote by $\phi_i$ the angle between the sides marked by $i-1$ and $i (\mod\
4)$. More exactly, $\phi_i$ are the turning angles for the velocity vector of a point that
runs along the perimeter in the direction given by the cyclic order $(1,2,3,4)$, see Figure
\ref{fig:SidesAngles}, left.

\begin{figure}[ht]
\centering
\begin{picture}(0,0)%
\includegraphics{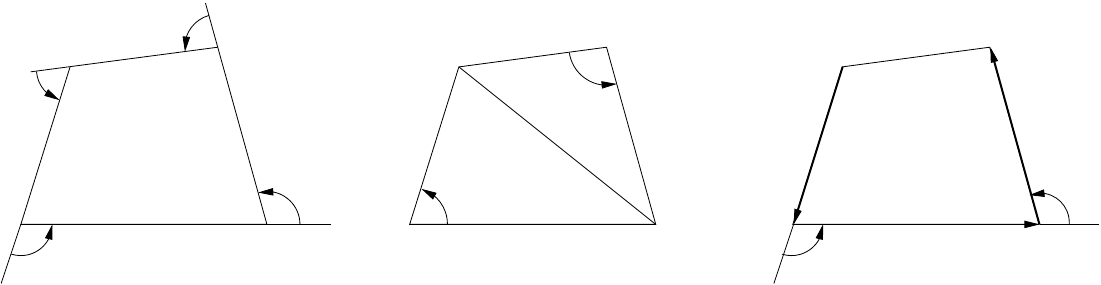}%
\end{picture}%
\setlength{\unitlength}{2072sp}%
\begingroup\makeatletter\ifx\SetFigFont\undefined%
\gdef\SetFigFont#1#2#3#4#5{%
  \reset@font\fontsize{#1}{#2pt}%
  \fontfamily{#3}\fontseries{#4}\fontshape{#5}%
  \selectfont}%
\fi\endgroup%
\begin{picture}(10059,2588)(-191,-1333)
\put(2044,144){\makebox(0,0)[lb]{\smash{{\SetFigFont{9}{10.8}{\rmdefault}{\mddefault}{\updefault}{\color[rgb]{0,0,0}$a_2$}%
}}}}
\put(7156,-1169){\makebox(0,0)[lb]{\smash{{\SetFigFont{9}{10.8}{\rmdefault}{\mddefault}{\updefault}{\color[rgb]{0,0,0}$\phi_1$}%
}}}}
\put(3511,-16){\makebox(0,0)[lb]{\smash{{\SetFigFont{9}{10.8}{\rmdefault}{\mddefault}{\updefault}{\color[rgb]{0,0,0}$a_4$}%
}}}}
\put(9136,164){\makebox(0,0)[lb]{\smash{{\SetFigFont{9}{10.8}{\rmdefault}{\mddefault}{\updefault}{\color[rgb]{0,0,0}$v_2$}%
}}}}
\put(6976, 29){\makebox(0,0)[lb]{\smash{{\SetFigFont{9}{10.8}{\rmdefault}{\mddefault}{\updefault}{\color[rgb]{0,0,0}$v_4$}%
}}}}
\put(1126,-961){\makebox(0,0)[lb]{\smash{{\SetFigFont{9}{10.8}{\rmdefault}{\mddefault}{\updefault}{\color[rgb]{0,0,0}$a_1$}%
}}}}
\put(4681,-961){\makebox(0,0)[lb]{\smash{{\SetFigFont{9}{10.8}{\rmdefault}{\mddefault}{\updefault}{\color[rgb]{0,0,0}$a_1$}%
}}}}
\put(8056,-1006){\makebox(0,0)[lb]{\smash{{\SetFigFont{9}{10.8}{\rmdefault}{\mddefault}{\updefault}{\color[rgb]{0,0,0}$v_1$}%
}}}}
\put(9541,-601){\makebox(0,0)[lb]{\smash{{\SetFigFont{9}{10.8}{\rmdefault}{\mddefault}{\updefault}{\color[rgb]{0,0,0}$\phi_2$}%
}}}}
\put(633,757){\makebox(0,0)[lb]{\smash{{\SetFigFont{9}{10.8}{\rmdefault}{\mddefault}{\updefault}{\color[rgb]{0,0,0}$a_3$}%
}}}}
\put(-166,-285){\makebox(0,0)[lb]{\smash{{\SetFigFont{9}{10.8}{\rmdefault}{\mddefault}{\updefault}{\color[rgb]{0,0,0}$a_4$}%
}}}}
\put(2521,-561){\makebox(0,0)[lb]{\smash{{\SetFigFont{9}{10.8}{\rmdefault}{\mddefault}{\updefault}{\color[rgb]{0,0,0}$\phi_2$}%
}}}}
\put(131,-1200){\makebox(0,0)[lb]{\smash{{\SetFigFont{9}{10.8}{\rmdefault}{\mddefault}{\updefault}{\color[rgb]{0,0,0}$\phi_1$}%
}}}}
\put(-52,300){\makebox(0,0)[lb]{\smash{{\SetFigFont{9}{10.8}{\rmdefault}{\mddefault}{\updefault}{\color[rgb]{0,0,0}$\phi_4$}%
}}}}
\put(1222,1054){\makebox(0,0)[lb]{\smash{{\SetFigFont{9}{10.8}{\rmdefault}{\mddefault}{\updefault}{\color[rgb]{0,0,0}$\phi_3$}%
}}}}
\put(3906,-649){\makebox(0,0)[lb]{\smash{{\SetFigFont{9}{10.8}{\rmdefault}{\mddefault}{\updefault}{\color[rgb]{0,0,0}$\pi-\phi_1$}%
}}}}
\put(4554,-37){\makebox(0,0)[lb]{\smash{{\SetFigFont{9}{10.8}{\rmdefault}{\mddefault}{\updefault}{\color[rgb]{0,0,0}$y$}%
}}}}
\put(4392,838){\makebox(0,0)[lb]{\smash{{\SetFigFont{9}{10.8}{\rmdefault}{\mddefault}{\updefault}{\color[rgb]{0,0,0}$a_3$}%
}}}}
\put(5710,-225){\makebox(0,0)[lb]{\smash{{\SetFigFont{9}{10.8}{\rmdefault}{\mddefault}{\updefault}{\color[rgb]{0,0,0}$a_2$}%
}}}}
\put(4525,360){\makebox(0,0)[lb]{\smash{{\SetFigFont{9}{10.8}{\rmdefault}{\mddefault}{\updefault}{\color[rgb]{0,0,0}$\pi-\phi_3$}%
}}}}
\end{picture}%
\caption{Notations; finding a relation for opposite
angles; finding a relation for adjacent angles.}
\label{fig:SidesAngles}
\end{figure}
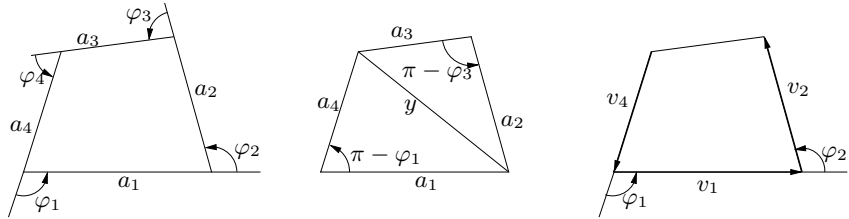

\subsection{Equations relating the angles of a quadrilateral}
\begin{lem}
\label{lem:CosLinDep}
The cosines of opposite angles of a quadrilateral are subject to a linear dependence, with coefficients depending on the side lengths.
\end{lem}
\begin{proof}
By expressing the diagonal length $y$ on Figure \ref{fig:SidesAngles} with the
help of the cosine
law first through $a_1, a_2, \phi_1$ and then through $a_3, a_4, \phi_3$,
we obtain
\begin{equation}
\label{eqn:LinCos}
a_1^2 + a_4^2 + 2a_1a_4 \cos\phi_1 = a_2^2 + a_3^2 + 2a_2a_3 \cos\phi_3
= 0
\end{equation}
\end{proof}
This result is classical. Bricard \cite{Bri97} mentions it as well-known. The following substitution also appears in \cite{Bri97}.
\begin{equation}
\label{eqn:XTan}
z_i := \tan \frac{\phi_i}2
\end{equation}

\begin{prp}
\label{prp:OppAngles}
The tangents of the opposite half-angles of a quadrilateral satisfy the following algebraic relation
\begin{equation}
\label{eqn:OppX}
b_{22} z_1^2 z_3^2 + b_{20} z_1^2 + b_{02} z_3^2 + b_{00} = 0, \quad
\text{where}
\end{equation}
\begin{equation*}
\begin{aligned}
b_{22} &= (a_1 + a_2 - a_3 - a_4)(a_1 - a_2 + a_3 - a_4)\\
b_{20} &= (a_1 + a_2 + a_3 - a_4)(a_1 - a_2 - a_3 - a_4)\\
b_{02} &= (a_1 - a_2 + a_3 + a_4)(a_1 + a_2 - a_3 + a_4)\\
b_{00} &= (a_1 - a_2 - a_3 + a_4)(a_1 + a_2 + a_3 + a_4)
\end{aligned}
\end{equation*}
\end{prp}
\begin{proof}
Substitute $\cos\phi_1 = \frac{1 - z_1^2}{1 + z_1^2}$
and $\cos\phi_3 = \frac{1 - z_3^2}{1 + z_3^2}$ in \eqref{eqn:LinCos}. A simple computation yields
\eqref{eqn:OppX}.
\end{proof}

With a bit more work one finds a relation between pairs of adjacent angles.
\begin{prp}
\label{prp:AdjCoord}
The tangents of the adjacent half-angles of a quadrilateral satisfy the equation
\begin{equation}
\label{eqn:AdjX}
c_{22} z_1^2 z_2^2 + c_{20} z_1^2 + c_{02} z_2^2 + 2 c_{11} z_1 z_2 + c_{00} = 0,
\quad \text{where}
\end{equation}
\begin{equation*}
\begin{aligned}
c_{22} &= (a_1 - a_2 - a_3 - a_4)(a_1 - a_2 + a_3 - a_4)\\
c_{20} &= (a_1 + a_2 + a_3 - a_4)(a_1 + a_2 - a_3 - a_4)\\
c_{02} &= (a_1 - a_2 + a_3 + a_4)(a_1 - a_2 - a_3 + a_4)\\
c_{11} &= -4a_2a_4\\
c_{00} &= (a_1 + a_2 - a_3 + a_4)(a_1 + a_2 + a_3 + a_4)
\end{aligned}
\end{equation*}
\end{prp}
\begin{proof}
Denote by $v_i$ the vector running along the $i$-th side of the quadrilateral, see Figure \ref{fig:SidesAngles}, right. Expanding the scalar
product on the left
hand side of $\|v_4 + v_1 + v_2\|^2 = \|v_3\|^2$ we obtain
$$
(a_1^2 + a_2^2 - a_3^2 + a_4^2) + 2 a_1 a_4 \cos \phi_1 + 2 a_1 a_2 \cos
\phi_2 + 2 a_2 a_4
\cos (\phi_1 + \phi_2) = 0
$$
The substitution and a tedious computation produce \eqref{eqn:AdjX}.
\end{proof}

\subsection{Bihomogeneous equations and algebraic curves in $\CP^1 \times \CP^1$}
The substitution \eqref{eqn:XTan} identifies $\R/2\pi$, which is the range of $\phi_i$, with $\RP^1$, which is the range of $z_i$. Thus it is geometrically reasonable to consider equations \eqref{eqn:OppX} and \eqref{eqn:AdjX} as equations on $(\RP^1)^2$. The proper setting for this are bihomogeneous polynomials.
That is, we introduce projective variables $(z_1 : w_1)$ and $(z_2 : w_2)$ and rewrite equation \eqref{eqn:AdjX} as
\begin{equation}
\label{eqn:Bihom}
c_{22} z_1^2 z_2^2 + c_{20} z_1^2 w_2^2 + c_{02} w_1^2 z_2^2 + 2 c_{11} z_1w_1 z_2w_2 + c_{00} w_1^2 w_2^2 = 0
\end{equation}
This point of view doesn't affect the affine part of the curve but may change the number of points at infinity (the infinity of $\RP^1 \times \RP^1$ is the union of two projective lines instead of one for $\RP^2$). Indeed, while the usual projectivization of \eqref{eqn:AdjX} has two points $(1 : 0 : 0)$ and $(0 : 1 : 0)$ at infinity, the curve \eqref{eqn:Bihom} has four:
\[
\left( \pm
\sqrt{\frac{-c_{02}}{c_{22}}}, \infty \right), \quad \left( \infty, \pm
\sqrt{\frac{c_{20}}{c_{22}}}
\right)
\]
For a generic choice of coefficients $c_{ij}$ the curve \eqref{eqn:Bihom} is non-singular, as opposed to the usual projectivization of \eqref{eqn:AdjX}.

%

\subsection{A system of six equations}
By Propositions \ref{prp:OppAngles} and \ref{prp:AdjCoord}, the angles of every quadrilateral satisfy a system of six equations: two of the form \eqref{eqn:OppX} and four of the form \eqref{eqn:AdjX}. In this section we show that, vice versa, under a certain genericity assumption every solution of the system corresponds to a quadrilateral.

\begin{prp}
Assume that the quadruple $a$ is not made of two pairs of equal adjacent numbers:
\begin{equation}
\label{eqn:NonDeg}
\text{neither } a_1 = a_2, a_3 = a_4 \text{ nor } a_1 = a_4, a_2 = a_3
\end{equation}
Then every solution $(z_1, z_2, z_3, z_4)$ of the system of six equations on the pairs of angles corresponds to a unique quadrilateral in $\Qor(a)$.
\end{prp}
\begin{proof}
By reverting the argument in the proof of Proposition \ref{prp:AdjCoord} one sees that for every solution of \eqref{eqn:AdjX} there is a unique quadrilateral in $\Qor(a)$ with angles $\phi_1$ and $\phi_2$. The same is true for the other three equations relating adjacent angles. For every solution of \eqref{eqn:OppX} there is also a quadrilateral with angles $\phi_1$ and $\phi_3$. This quadrilateral can be non-unique only if the second and the fourth vertices coincide, for which $a_1=a_4$ and $a_2=a_3$ is needed.

Thus for every solution $(z_1, z_2, z_3, z_4)$ of the system of six equations there are six quadrilaterals $C_{ij} \in \Qor(a)$ with the property that in $C_{ij}$ the angles $\phi_i$ and $\phi_j$ have the correct values. Our goal is to show that there is $C$ with all correct angles. Consider $C_{12}$, $C_{13}$, and $C_{14}$. They all have the same $\phi_1$. Assumption \eqref{eqn:NonDeg} implies that there are at most two quadrilaterals in $\Qor(a)$ with a given $\phi_1$. Thus at least two of the three quadrilaterals must coincide. This yields a quadrilateral $C_1$ with three correct angles, one of which is $\phi_1$. If the fourth angle $\phi_i$ is incorrect, then again, consider the three quadrilaterals where this angle is correct and obtain a quadrilateral $C_i$ with three correct angles, one of which is $\phi_i$. As $C_1$ and $C_i$ have two angles in common, they coincide, and $C = C_1 = C_i$ is the desired quadrilateral.
\end{proof}

\subsection{Birational equivalence between \eqref{eqn:OppX} and \eqref{eqn:AdjX}}
We have just seen that the configuration space $\Qor(a)$ is an algebraic curve in $(\RP^1)^4$ given by six equations. It will often be convenient to restrict our attention to a single equation in two variables. The following statement allows us to do so.
\begin{prp}
\label{prp:AlgConfSpace}
The projection of $\Qor(a) \subset (\RP^1)^4$ to every coordinate plane $(z_i, z_{i+1})$ is
a birational equivalence.

The projection of $\Qor(a)$ to the coordinate plane $(z_i, z_{i+2})$ is a birational
equivalence unless $a_i = a_{i-1}$ and $a_{i+1} = a_{i+2}$. (All indices are taken in $\{1,2,3,4\}$ modulo $4$.)
\end{prp}
\begin{proof}
It suffices to consider the case $i=1$.
As noted before, the projection of $\Qor(a)$ to the $(z_1, z_2)$-plane is injective. It suffices to show that for every $(z_1, z_2, z_3, z_4) \in \Qor(a)$ the values of $z_3$ and $z_4$ are rational functions of $z_1$ and $z_2$. By projecting the quadrilateral to its second side and to the line orthogonal to it, we obtain
\begin{equation}
\label{eqn:Birat}
\begin{aligned}
a_1 \cos\phi_2 + a_2 + a_3 \cos\phi_3 + a_4 \cos(\phi_1+\phi_2) = 0\\
a_1 \sin\phi_2 - a_3 \sin\phi_3 + a_4 \sin(\phi_1+\phi_2) = 0
\end{aligned}
\end{equation}
It follows that
\[
z_3 = \frac{1-\cos\phi_3}{\sin\phi_3} = \frac{a_1 \cos\phi_2 + a_2 + a_3 + a_4 \cos(\phi_1+\phi_2)}{a_1 \sin\phi_2 + a_4 \sin(\phi_1+\phi_2)}
\]
which is rational in $z_1$ and $z_2$.

Similarly, if either $a_1 \ne a_4$ or $a_2 \ne a_3$, then the projection of $\Qor(a)$ to the $(z_1, z_3)$-plane is injective. To show that the inverse map is rational, rewrite \eqref{eqn:Birat} as
\[
\begin{aligned}
(a_1 + a_4 \cos\phi_1) \cos\phi_2 - a_4 \sin\phi_1 \cdot \sin\phi_2 = -a_2 - a_3\cos\phi_3\\
a_4\sin\phi_1 \cdot \cos\phi_2 + (a_1 + a_4 \cos\phi_1) \sin\phi_2 = a_3 \sin\phi_3
\end{aligned}
\]
This is a system of linear equations on $\cos\phi_2$ and $\sin\phi_2$ with the determinant $(a_1 + a_4 \cos\phi_1 )^2 + a_4^2 \sin^2\phi_1$. The determinant vanishes only for $\phi_1 = \pi$ and $a_1 = a_4$, which also implies $a_2 = a_3$. If it does not vanish, then by solving the system we can express $\cos\phi_2$ and $\sin\phi_2$ as rational functions of $z_1$ and $z_3$.
\end{proof}

Now that the algebraic structure of $\Qor(a)$ became clear, we can define the complexified configuration space.

\begin{dfn}
The \emph{complexified configuration space} $\QorC(a)$ of oriented quadrilaterals with edge lengths $a$ is the algebraic curve in $(\CP^1)^4$ defined by the six equations of the form \eqref{eqn:OppX} and \eqref{eqn:AdjX}.
\end{dfn}

Due to Proposition \ref{prp:AlgConfSpace}, $\QorC(a)$ is birationally equivalent to the curve \eqref{eqn:AdjX} and is birationally equivalent to \eqref{eqn:OppX} unless $a_1 = a_4$ and $a_2 = a_3$.

\section{Parametrizations of the configuration spaces of oriented quadrilaterals}
\label{sec:Param}
\subsection{Classifying linkages by their degree of degeneracy}
The shape of the configuration space $\QorC(a)$ turns out to depend on the number of solutions of the equation
\begin{equation}
\label{eqn:Grashof}
a_1 \pm a_2 \pm a_3 \pm a_4 = 0
\end{equation}

It is easy to see that if \eqref{eqn:Grashof} has at least two solutions, then $(a_1, a_2, a_3, a_4)$ consists of two pairs of equal numbers. This leads to the following classification of quadrilaterals.

\begin{dfn}
\label{dfn:Types}
A quadrilateral with side lengths $a_1, a_2, a_3, a_4$ (in this cyclic order) is said to be
\begin{itemize}
\item of \emph{elliptic} type, if equation \eqref{eqn:Grashof} has no solution;
\item of \emph{conic} type, if the equation \eqref{eqn:Grashof} has exactly one solution;
\item an \emph{isogram}, if its opposite sides are equal: $a_1 = a_3$, $a_2 = a_4$;
\item a \emph{deltoid}, if two pairs of adjacent sides are equal: $a_1 = a_2, a_3 = a_4$ or $a_1 = a_4, a_2 = a_3$;
\item a \emph{rhombus}, if all sides are equal.
\end{itemize}
\end{dfn}

Let us look at the last three simple cases before we proceed to the more interesting conic and elliptic types.

\subsubsection{The rhombus}
Equation \eqref{eqn:AdjX} becomes $z_1^2 = z_3^2$, and \eqref{eqn:OppX} becomes (in the bihomogeneous form)
\[
w_1w_2(z_1z_2 - w_1w_2) = 0
\]
The other equations can be obtained by cyclically permuting the indices. The configuration space consists of three lines
\[
z_1 = \infty, \quad z_2 = \infty, \quad z_1z_2 = 1
\]
The first two consist of ``folded'' configurations, when two opposite vertices are at the same point, and the edges rotate around this point; the third line corresponds to actual rhombi with $\phi_1 + \phi_2 = \pi$.

Note that the condition \eqref{eqn:NonDeg} is violated, and the system has ``fantom'' solutions that don't correspond to any quadrilateral, e.~g. $z_1 = z_3 = \infty$, $z_2 = z_4$.

\subsubsection{The deltoid}
Assume $a_1 = a_2 \ne a_3 = a_4$. Equation \eqref{eqn:AdjX} becomes
\[
(a_1-a_3)z_1^2 - 2a_3z_1z_2 + (a_1+a_3) = 0
\]
Besides, in the bihomogeneous form the factor $w_2$ appears, so that the configuration space consists of two lines
\[
\QorC(a) = \{z_2 = \infty\} \cup \left\{ z_2 = \frac{(a_1-a_3)z_1^2 + (a_1+a_3)}{2a_3z_1} \right\}
\]
Again, the first line corresponds to a folded deltoid, and the second expresses an angle at a ``peak'' of the deltoid through the angle at its base. The fantom solution  $z_1 = z_3 = \infty$, $z_2 = z_4$ is also present in this case.

\subsubsection{The isogram}
Under assumption $a_1 = a_3 \ne a_2 = a_4$ equation \eqref{eqn:AdjX} becomes
\[
(a_1-a_2)z_1^2z_2^2 + 2a_2 z_1z_2 - (a_1+a_2) = 0
\]
which factorizes as
\[
(z_1z_2 - 1)((a_1-a_2)z_1z_2 + (a_1+a_2)) = 0
\]
The configuration space consists of two lines: $z_2 = z_1^{-1}$, consisting of parallelograms, and $z_2 = \frac{a_1+a_2}{a_1-a_2} z_2^{-1}$, consisting of antiparallelograms.

\subsection{Conic quadrilaterals: parametrization by trigonometric functions}
If equation \eqref{eqn:Grashof} has exactly one solution, then two cases must be distinguished: either the two pairs of opposite sides add up to the same total length, or two pairs of adjacent sides do so.

In what follows we will be repeatedly using the fact that $a+b = c+d$ implies
\begin{equation*}
ac - bd = (a+b)(a-d), \quad cd - ab = (a-c)(a-d)
\end{equation*}

\subsubsection{Circumsribable quadrilaterals: $a_1 + a_3 = a_2 + a_4$}

\begin{prp}
\label{prp:Param1+3=2+4}
Assume that $a_1 + a_3 = a_2 + a_4$ is the unique solution of the equation \eqref{eqn:Grashof} and that $a_3 = a_{\min}$. Then the configuration space $\QorC(a)$ is isomorphic to the one-point compactification of $\C/2\pi\Z$, and a bijection $\C/2\pi\Z \to \QorC(a) \setminus \{\infty\}$ can be established by the parametrization
\begin{gather*}
z_1 = p_1 \sin t, \quad z_2 = p_2 \sin (t + i\sigma), \\
z_3 = p_3 \sin \left(t - \frac{\pi}2\right), \quad
z_4 = p_4 \sin \left(t + \frac{\pi}2 + i\sigma\right)
\end{gather*}
where
$$
p_1 = \sqrt{\frac{(a_1+a_3)(a_1-a_2)}{a_2a_3}} = \sqrt{\frac{a_1a_4}{a_2a_3} - 1}
$$
with $p_2$, $p_3$, $p_4$ obtained by cyclically permuting the indices, and
$$
\sigma = \ln \left( \sqrt{\frac{a_2a_4}{a_2a_4 - a_1a_3}} + \sqrt{\frac{a_1a_3}{a_2a_4 - a_1a_3}} \right)
$$
The square roots of negative numbers are assumed to take value in $i\R_+$.
\end{prp}
\begin{proof}
Due to $a_1+a_3 = a_2+a_4$ equations \eqref{eqn:OppX} and \eqref{eqn:AdjX} become
\begin{subequations}
\begin{equation}
\label{eqn:OppCon}
-a_2a_3 z_1^2 + a_1a_4 z_3^2 + (a_1a_4 - a_2a_3) = 0
\end{equation}
\begin{equation}
\label{eqn:Conic1}
a_2(a_1-a_4) z_1^2 + a_4(a_1-a_2) z_2^2 - 2a_2a_4 z_1z_2 + a_1(a_1+a_3) = 0
\end{equation}
\end{subequations}
Equation \eqref{eqn:OppCon} rewrites as
\[
\frac{z_1^2}{p_1^2} + \frac{z_3^2}{p_3^2} = 1
\]
with
\[
p_1 = \sqrt{\frac{a_1a_4}{a_2a_3} - 1}, \quad p_3 = \sqrt{\frac{a_2a_3}{a_1a_4} - 1}
\]
so that the coordinates $(z_1, z_3)$ can be parametrized as
\[
z_1 = p_1 \sin t, \quad z_3 = p_3 \cos t
\]


The substitution of $z_1 = p_1 u_1$ and $z_2 = p_2 u_2$ in \eqref{eqn:Conic1} under the assumption $a_1 = a_{\max}$ results in
\[
u_1^2 + u_2^2 - 2 \sqrt{\frac{a_2a_4}{a_2a_4-a_1a_3}} u_1u_2 + \frac{a_1a_3}{a_2a_4 - a_1a_3} = 0
\]
which has the form
\begin{equation}
\label{eqn:SinShift}
u^2 + v^2 - 2 \cos \tau uv - \sin^2 \tau = 0
\end{equation}
with $\tau = i\sigma$ and $\sigma$ as stated in the theorem. On the other hand, $u$ and $v$ in \eqref{eqn:SinShift} can be parametrized as $u(t) = \sin t$ and $v(t) = \sin(t+\tau)$. This leads to
\[
z_1 = p_1 \sin t, \quad z_2 = p_2 \sin(t+i\sigma)
\]
The relation between each pair of opposite angles imply that the above parametrization is completed by
\[
z_3 = p_3 \sin\left(t \pm \frac{\pi}2\right), \quad z_4 = p_4 \sin\left(t+i\sigma \pm \frac{\pi}2\right)
\]
for some choices of the signs. To determine the signs, one may analyze the equations between $z_2$ and $z_3$ and between $z_3$ and $z_4$ in a similar way. Alternatively, one looks at special configurations of quadrilaterals.

\end{proof}

\subsubsection{Edge lengths satisfy $a_1+a_2 = a_3+a_4$}
Equation \eqref{eqn:OppX} relating $z_1$ and $z_3$ takes the same form as in the case of curcumscribable quadrilaterals:
\[
-a_2a_3 z_1^2 + a_1a_4 z_3^2 + (a_1a_4 - a_2a_3) = 0
\]
so that we have again a parametrization
\[
z_1 = p_1 \sin t, \quad z_3 = p_3 \cos t
\]

But for the pair $\phi_2, \phi_4$ we have
\[
(a_3a_4-a_1a_2) z_2^2z_4^2 - a_1a_2 z_2^2 + a_3a_4 z_4^2 = 0
\]
or, equivalently,
\[
-a_3a_4 z_2^{-2} + a_1a_2 z_4^{-2} + (a_1a_2 - a_3a_4) = 0
\]
which leads to
\[
z_2^{-1} = p_2 \sin t, \quad z_4^{-1} = p_4 \cos t
\]
Similarly to the preceding section, we obtain the following (note a sign change for $z_3$).
\begin{prp}
\label{prp:Param1+2=3+4}
Assume that $a_1 + a_2 = a_3 + a_4$ is the unique solution of the equation \eqref{eqn:Grashof} and that $a_3 = a_{\min}$. Then the configuration space $\QorC(a)$ is isomorphic to a one-point compactification of $\C/2\pi\Z$, and the bijection $\C/2\pi\Z \to \QorC(a) \setminus \{\infty\}$ can be established by the parametrization
\begin{gather*}
z_1 = p_1 \sin t, \quad z_2^{-1} = p_2 \sin (t + i\sigma), \\
z_3 = p_3 \sin \left(t + \frac{\pi}2\right), \quad
z_4^{-1} = p_4 \sin \left(t + \frac{\pi}2 + i\sigma\right)
\end{gather*}
where
$$
p_1 = \sqrt{\frac{a_1a_4}{a_2a_3} - 1}
$$
with $p_2$, $p_3$, $p_4$ obtained by cyclically permuting the indices, and
$$
\sigma = \ln \left( \sqrt{\frac{a_2a_4}{a_2a_4 - a_1a_3}} + \sqrt{\frac{a_1a_3}{a_2a_4 - a_1a_3}} \right)
$$
\end{prp}

\subsubsection{The real part of the configuration space}
In Theorems \ref{prp:Param1+3=2+4} and \ref{prp:Param1+2=3+4} we have $p_1, p_2 \in \R$ and $p_3,p_4 \in i\R$. It follows that the variables $z_i$ take real values if and only if $\Re t \equiv \frac{\pi}2 (\mod \pi)$. The parameter change $t = \frac\pi{2} + iu$ leads to a parametrization that makes the real part of $\QorC$ more tangible.

\begin{cor}
Assume that $a_1 + a_3 = a_2 + a_4$ is the unique solution of the equation \eqref{eqn:Grashof} and that $a_3 = a_{\min}$. Then the configuration space $\QorC$ can be parametrized as
\begin{equation}
\begin{aligned}
z_1 &= p_1 \cosh u\\
z_2 &= p_2 \cosh (u+\sigma)\\
z_3 &= ip_3 \sinh u\\
z_4 &= -ip_4 \sinh (u+\sigma)
\end{aligned}
\end{equation}
where
$$
p_1 = \sqrt{\frac{a_1a_4}{a_2a_3} - 1} \quad \text{etc.}
$$
and
$$
\cosh\sigma = \sqrt{\frac{a_2a_4}{a_2a_4 - a_1a_3}}, \quad \sinh\sigma = \sqrt{\frac{a_1a_3}{a_2a_4 - a_1a_3}}
$$
See Figure \ref{fig:RealConic}.

If $a_1 + a_2 = a_3 + a_4$ is the unique solution of \eqref{eqn:Grashof}, and $a_3 = a_{\min}$, then there is a parametrization
\begin{equation}
\begin{aligned}
z_1 &= p_1 \cosh u\\
z_2^{-1} &= p_2 \cosh(u+\sigma)\\
z_3 &= -ip_3 \sinh u\\
z_4^{-1} &= -ip_4 \sinh(u+\sigma)
\end{aligned}
\end{equation}
with $p_i$ and $\sigma$ computed by the same formulas. See Figure \ref{fig:RealConic2}.
\end{cor}

\begin{figure}[ht]
\centering
\begin{picture}(0,0)%
\includegraphics{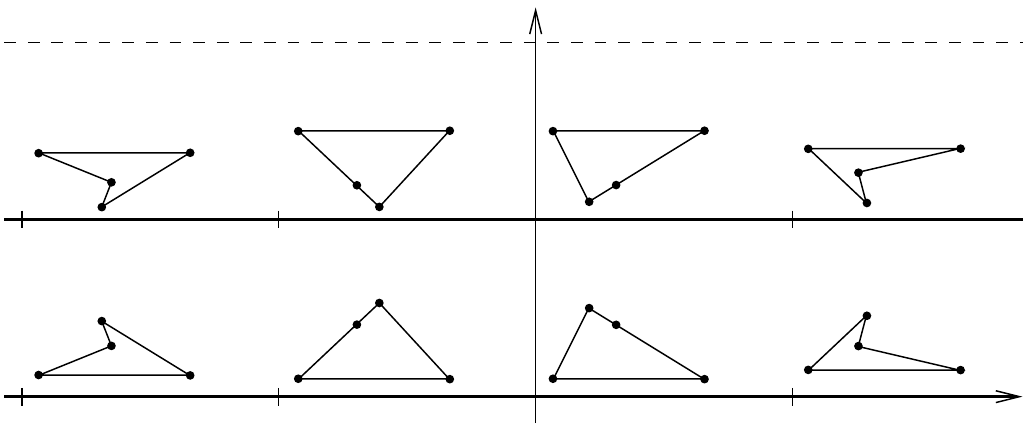}%
\end{picture}%
\setlength{\unitlength}{3729sp}%
\begingroup\makeatletter\ifx\SetFigFont\undefined%
\gdef\SetFigFont#1#2#3#4#5{%
  \reset@font\fontsize{#1}{#2pt}%
  \fontfamily{#3}\fontseries{#4}\fontshape{#5}%
  \selectfont}%
\fi\endgroup%
\begin{picture}(5219,2203)(-21,-1151)
\put(5086,-871){\makebox(0,0)[lb]{\smash{{\SetFigFont{8}{9.6}{\rmdefault}{\mddefault}{\updefault}{\color[rgb]{0,0,0}$\Re$}%
}}}}
\put(2746,-1096){\makebox(0,0)[lb]{\smash{{\SetFigFont{8}{9.6}{\rmdefault}{\mddefault}{\updefault}{\color[rgb]{0,0,0}$0$}%
}}}}
\put(4051,-1051){\makebox(0,0)[lb]{\smash{{\SetFigFont{8}{9.6}{\rmdefault}{\mddefault}{\updefault}{\color[rgb]{0,0,0}$\sigma$}%
}}}}
\put(1441,-1051){\makebox(0,0)[lb]{\smash{{\SetFigFont{8}{9.6}{\rmdefault}{\mddefault}{\updefault}{\color[rgb]{0,0,0}$-\sigma$}%
}}}}
\put(136,-1051){\makebox(0,0)[lb]{\smash{{\SetFigFont{8}{9.6}{\rmdefault}{\mddefault}{\updefault}{\color[rgb]{0,0,0}$-2\sigma$}%
}}}}
\put(2746,-196){\makebox(0,0)[lb]{\smash{{\SetFigFont{8}{9.6}{\rmdefault}{\mddefault}{\updefault}{\color[rgb]{0,0,0}$\pi i$}%
}}}}
\put(4006,-196){\makebox(0,0)[lb]{\smash{{\SetFigFont{8}{9.6}{\rmdefault}{\mddefault}{\updefault}{\color[rgb]{0,0,0}$\sigma+\pi i$}%
}}}}
\put(136,-151){\makebox(0,0)[lb]{\smash{{\SetFigFont{8}{9.6}{\rmdefault}{\mddefault}{\updefault}{\color[rgb]{0,0,0}$-2\sigma + \pi i$}%
}}}}
\put(1441,-151){\makebox(0,0)[lb]{\smash{{\SetFigFont{8}{9.6}{\rmdefault}{\mddefault}{\updefault}{\color[rgb]{0,0,0}$-\sigma + \pi i$}%
}}}}
\put(2746,749){\makebox(0,0)[lb]{\smash{{\SetFigFont{8}{9.6}{\rmdefault}{\mddefault}{\updefault}{\color[rgb]{0,0,0}$2\pi i$}%
}}}}
\put(2476,929){\makebox(0,0)[lb]{\smash{{\SetFigFont{8}{9.6}{\rmdefault}{\mddefault}{\updefault}{\color[rgb]{0,0,0}$\Im$}%
}}}}
\end{picture}%
\caption{The real part of the configuration space in the case $a_1 + a_3 = a_2 + a_4$.}
\label{fig:RealConic}
\end{figure}

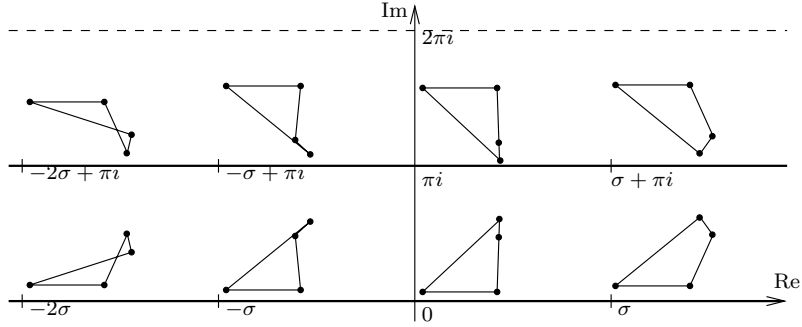
\begin{figure}[ht]
\centering
\begin{picture}(0,0)%
\includegraphics{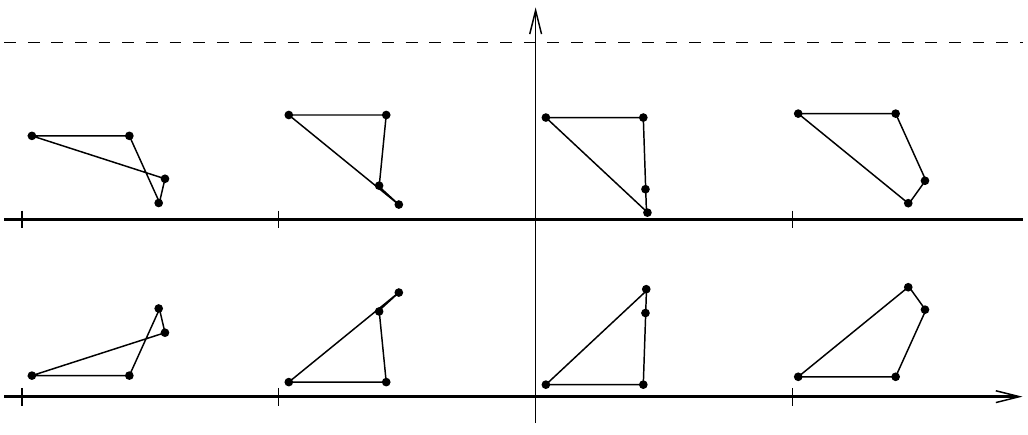}%
\end{picture}%
\setlength{\unitlength}{3729sp}%
\begingroup\makeatletter\ifx\SetFigFont\undefined%
\gdef\SetFigFont#1#2#3#4#5{%
  \reset@font\fontsize{#1}{#2pt}%
  \fontfamily{#3}\fontseries{#4}\fontshape{#5}%
  \selectfont}%
\fi\endgroup%
\begin{picture}(5219,2203)(-21,-1151)
\put(5086,-871){\makebox(0,0)[lb]{\smash{{\SetFigFont{8}{9.6}{\rmdefault}{\mddefault}{\updefault}{\color[rgb]{0,0,0}$\Re$}%
}}}}
\put(2746,-1096){\makebox(0,0)[lb]{\smash{{\SetFigFont{8}{9.6}{\rmdefault}{\mddefault}{\updefault}{\color[rgb]{0,0,0}$0$}%
}}}}
\put(4051,-1051){\makebox(0,0)[lb]{\smash{{\SetFigFont{8}{9.6}{\rmdefault}{\mddefault}{\updefault}{\color[rgb]{0,0,0}$\sigma$}%
}}}}
\put(1441,-1051){\makebox(0,0)[lb]{\smash{{\SetFigFont{8}{9.6}{\rmdefault}{\mddefault}{\updefault}{\color[rgb]{0,0,0}$-\sigma$}%
}}}}
\put(136,-1051){\makebox(0,0)[lb]{\smash{{\SetFigFont{8}{9.6}{\rmdefault}{\mddefault}{\updefault}{\color[rgb]{0,0,0}$-2\sigma$}%
}}}}
\put(2746,-196){\makebox(0,0)[lb]{\smash{{\SetFigFont{8}{9.6}{\rmdefault}{\mddefault}{\updefault}{\color[rgb]{0,0,0}$\pi i$}%
}}}}
\put(4006,-196){\makebox(0,0)[lb]{\smash{{\SetFigFont{8}{9.6}{\rmdefault}{\mddefault}{\updefault}{\color[rgb]{0,0,0}$\sigma+\pi i$}%
}}}}
\put(136,-151){\makebox(0,0)[lb]{\smash{{\SetFigFont{8}{9.6}{\rmdefault}{\mddefault}{\updefault}{\color[rgb]{0,0,0}$-2\sigma + \pi i$}%
}}}}
\put(1441,-151){\makebox(0,0)[lb]{\smash{{\SetFigFont{8}{9.6}{\rmdefault}{\mddefault}{\updefault}{\color[rgb]{0,0,0}$-\sigma + \pi i$}%
}}}}
\put(2746,749){\makebox(0,0)[lb]{\smash{{\SetFigFont{8}{9.6}{\rmdefault}{\mddefault}{\updefault}{\color[rgb]{0,0,0}$2\pi i$}%
}}}}
\put(2476,929){\makebox(0,0)[lb]{\smash{{\SetFigFont{8}{9.6}{\rmdefault}{\mddefault}{\updefault}{\color[rgb]{0,0,0}$\Im$}%
}}}}
\end{picture}%
\caption{The real part of the configuration space in the case $a_1 + a_2 = a_3 + a_4$.}
\label{fig:RealConic2}
\end{figure}

\subsection{Elliptic quadrilaterals: parametrization by elliptic functions}
\label{sec:EllParam}
\subsubsection{Special biquadratic equations and elliptic curves}
\label{sec:Modulus}
A biquadratic polynomial is a polynomial of degree $2$ in each of its two variables. Here we will look at a special class of biquadratic polynomials.
\begin{lem}
Let $m \in \C \setminus \{0, 1\}$. Then the equation
\begin{equation}
\label{eqn:Biquad}
u^2 + v^2 = 1 + mu^2v^2
\end{equation}
defines an elliptic curve.
\end{lem}
\begin{proof}
The substitution
\[
x = u,\, y = v(1-mu^2) \Leftrightarrow u = x,\, v = \frac{y}{1-mx^2}
\]
transforms equation \eqref{eqn:Biquad} into
$$
y^2 = (1 - x^2)(1 - m x^2)
$$
If $m \notin \{0,1\}$, then the polynomial on the right hand side has four distinct roots, thus the equation defines an elliptic curve.
\end{proof}

It is possible to transform equation \eqref{eqn:OppX} to the form \eqref{eqn:Biquad} by a substitution $z_1 = p_1 u_1, z_3 = p_3 u_3$. In order to make the formulas more compact we need some handy notations.

Let $a, b, c, d \in \R$. Put
\begin{equation}
\label{eqn:abar}
\begin{aligned}
s := \frac{a+b+c+d}2, \quad \bar a &:= s-a = \frac{-a+b+c+d}2\\
\bar b &:= s-b, \quad \bar c := s-c, \quad \bar d := s-d
\end{aligned}
\end{equation}

We have $a+b = \bar c + \bar d$, $s-a-b = -(s-c-d)$ etc. Other, less obvious identities, are collected in the following lemma.
\begin{lem}
\label{lem:SIdent}
In the above notations we have
\begin{equation*}
\begin{aligned}
ab - \bar c \bar d &= (s-a-c)(s-b-c)\\
ab - \bar a \bar b &= s(s-c-d)\\
abcd - \bar a \bar b \bar c \bar d &= s(s-a-b)(s-b-c)(s-a-c)
\end{aligned}
\end{equation*}
\end{lem}

Now equation \eqref{eqn:OppX} can be rewritten in any of the following two ways:
\begin{subequations}
\begin{equation}
\label{eqn:OppZinS}
(s - a_3 - a_4)(s - a_2 - a_4) z_1^2 z_3^2 - \bar a_1 \bar a_4 z_1^2 + \bar a_2 \bar a_3 z_3^2 + s(s - a_2 - a_3) = 0
\end{equation}
\begin{equation}
\label{eqn:OppZinSa}
(a_2a_3 - \bar a_1\bar a_4) z_1^2 z_3^2 - \bar a_1 \bar a_4 z_1^2 + \bar a_2 \bar a_3 z_3^2 + (a_1a_4 - \bar a_1 \bar a_4) = 0
\end{equation}
\end{subequations}

\begin{prp}
\label{prp:QEll}
The substitution $z_1 = p_1 u$, $z_3 = p_3 v$ with
\begin{equation}
\label{eqn:p1p3}
p_1 = \sqrt{\frac{s(s-a_2-a_3)}{\bar a_1 \bar a_4}} = \sqrt{\frac{a_1a_4}{\bar a_1 \bar a_4} - 1}, \quad p_3 = \sqrt{\frac{s(s-a_1-a_4)}{\bar a_2 \bar a_3}} = \sqrt{\frac{a_2a_3}{\bar a_2 \bar a_3} - 1}
\end{equation}
transforms equation \eqref{eqn:OppX} to \eqref{eqn:Biquad} with
\begin{equation}
\label{eqn:m}
m = -\frac{s(s-a_2-a_3)(s-a_2-a_4)(s-a_3-a_4)}{\bar a_1 \bar a_2 \bar a_3 \bar a_4} = 1 - \frac{a_1a_2a_3a_4}{\bar a_1 \bar a_2 \bar a_3 \bar a_4}
\end{equation}
In particular, if \eqref{eqn:Grashof} has no solutions, then $\QorC(a)$ is an elliptic curve.
\end{prp}
\begin{proof}
An equation of the form
\[
Az^2 w^2 + Bz^2 + Cw^2 + D = 0
\]
with $D \ne 0$ can be brought into the form \eqref{eqn:Biquad} by the substitution
\[
z^2 = -\frac{D}{B} u^2, \quad w^2 = -\frac{D}{C} v^2
\]
This yields $m = \frac{AD}{BC}$. Applying this to the equation \eqref{eqn:OppZinS} and taking into account the third identity of Lemma \ref{lem:SIdent} gives us the value of $m$ as stated in the proposition. Note that $B, C \ne 0$ due to \eqref{eqn:QuadIneqB}, and $D \ne 0$, since otherwise $a_1 + a_4 = a_2 + a_3$.

Further, $m \ne 1$ because $a_i \ne 0$, and $m \ne 0$ because this would result in $s = a_i + a_j$ for some $i, j$.
\end{proof}

\subsubsection{Biquadratic equation and Jacobi elliptic functions}
Jacobi's elliptic functions can be defined for any complex value $k \notin \{0, \pm 1\}$ of the Jacobi modulus, \cite{McKM99, AE06}.
\begin{lem}
\label{lem:EllParam}
Let $m \in \C \setminus \{0,1\}$. Then the curve
\begin{equation}
\label{eqn:SnShiftK}
x^2 + y^2 = m x^2 y^2 + 1
\end{equation}
possesses the parametrization
$$
x(t) = \sn (t;k), \qquad y(t) = \sn(t+K;k) \quad \text{with } k = \sqrt{m}
$$
as well as the parametrization
$$
x(t) = \cn (t;k), \qquad y(t) = \cn(t+K;k) \quad \text{with } k = \sqrt{\frac{m}{m-1}}
$$
\end{lem}
\begin{proof}
The $\sn$-parametrization follows from the identity
\[
\sn(t+K;k) = \frac{\cn(t;k)}{\dn(t;k)} = \pm \sqrt{\frac{1-\sn^2(t;k)}{1-k^2\sn^2(t;k)}}
\]
The $\cn$-parametrization follows from the identity
\[
\cn(t+K;k) = -k' \frac{\sn(t;k)}{\dn(t;k)} = \pm k' \sqrt{\frac{1-\cn^2(t;k)}{(k')^2 + k^2\cn^2(t;k)}}
\]
\end{proof}

\begin{rem}
By a parameter change or by symmetry reasons one can obtain alternative parametrizations of \eqref{eqn:SnShiftK}:
\[
x = \sn t, \quad y = \sn(t-K) \quad \text{and} \quad x = \cn t, \quad y = \cn(t-K)
\]
\end{rem}

The value of $m$ in our case is given by \eqref{eqn:m}, so that $m \in (-\infty, 0) \cup (0,1)$. Let us make a case distinction in order to bring the Jacobi parameter $k$ into its usual range $(0,1)$.

\begin{prp}
Let $a_{\min}$ and $a_{\max}$ denote the minimum, respectively the maximum value taken by $a_1, a_2, a_3, a_4$, and let $a_{\min} + a_{\max} \ne s$.
\begin{enumerate}
\item
If $a_{\min} + a_{\max} < s$, then the complexified configuration space $\QorC(a)$ is isomorphic to the quotient of $\C$ by a rectangular lattice, which is the period lattice of the elliptic sine function with the parameter
\[
k = \sqrt{1 - \frac{a_1a_2a_3a_4}{\bar a_1 \bar a_2 \bar a_3 \bar a_4}} \in (0,1)
\]
The tangents of halves of the opposite angles $\phi_1$ and $\phi_3$ can be parametrized as
\[
z_1 = p_1 \sn(t;k), \quad z_3 = p_3 \sn(t+K;k)
\]
with $p_1$ and $p_3$ as in \eqref{eqn:p1p3}.
\item
If $a_{\min} + a_{\max} > s$, then the complexified configuration space $\QorC(a)$ is isomorphic to the quotient of $C$ by a rhombic lattice, which is the period lattice of the elliptic cosine function with the parameter
\[
k = \sqrt{1 - \frac{\bar a_1 \bar a_2 \bar a_3 \bar a_4}{a_1a_2a_3a_4}} \in (0,1)
\]
The tangents of halves of the opposite angles $\phi_1$ and $\phi_3$ can be parametrized as
\[
z_1 = p_1 \cn(t;k), \quad z_3 = p_3 \cn(t+K;k)
\]
with $p_1$ and $p_3$ as in \eqref{eqn:p1p3}.
\end{enumerate}
\end{prp}
\begin{proof}
This is a direct consequence of Proposition \ref{prp:QEll} and Lemma \ref{lem:EllParam}. The only things to observe are that $a_{\min} + a_{\max} \ne s$ if and only if \eqref{eqn:Grashof} has no solutions, and
\begin{multline*}
a_{\min} + a_{\max} < s \Leftrightarrow s(s-a_1-a_2)(s-a_2-a_3)(s-a_3-a_1) < 0 \\
\Leftrightarrow a_1a_2a_3a_4 < \bar a_1 \bar a_2 \bar a_3 \bar a_4
\end{multline*}
\end{proof}

\subsubsection{Shift parametrizations of more general biquadratic equations}
A general symmetric biquadratic equation can be parametrized by a shift of an elliptic function: $x = f(t), y = f(t+\tau)$. This is essentially due to Euler \cite{Eul68} who used biquadratic equations to integrate the relation
$$
\frac{dx}{\sqrt{P(x)}} + \frac{dy}{\sqrt{P(y)}} = 0, \quad \deg P = 4
$$
Euler's argument in a special case is reproduced in \cite{Mar66}.
The following lemma is a special case that we need in order to parametrize equation \eqref{eqn:AdjX}.

\begin{lem}
\label{lem:SnShift}
The curve
\begin{equation}
\label{eqn:SnShift}
x^2 + y^2 = A x^2 y^2 + 2B xy + C
\end{equation}
possesses the parametrization
$$
x(t) = \sn(t;k), \qquad y(t) = \sn(t+\tau;k)
$$
if $k \in (0,1)$ and $\tau \in \C$ satisfy
\begin{equation}
\label{eqn:SnSubst}A = k^2 \sn^2\tau, \quad B = \cn\tau \dn\tau, \quad C = \sn^2\tau
\end{equation}
and the parametrization
$$
x(t) = \cn(t), \qquad y(t) = \cn(t+\tau)
$$
if $k \in (0,1)$ and $\tau \in \C$ satisfy
\begin{equation}
\label{eqn:CnSubst}
A = -k^2 \frac{\sn^2\tau}{\dn^2\tau}, \quad B = \frac{\cn\tau}{\dn^2\tau}, \quad C = (k')^2 \frac{\sn^2\tau}{\dn^2\tau}
\end{equation}
Here $k' = \sqrt{1-k^2}$ is the conjugate modulus.
\end{lem}
\begin{proof}
The addition law
$$
\sn(t+\tau) = \frac{\sn t \cn\tau \dn\tau + \cn t \dn t \sn\tau}{1 - k^2 \sn^2 t \sn^2\tau}
$$
can be rewritten as
$$
(1 - k^2 \sn^2 t \sn^2\tau) \sn(t+\tau) - \cn\tau \dn\tau \sn t = \sqrt{(1 - \sn^2 t)(1 - k^2 \sn^2 t)}
$$
By taking the squares of both parts and substituting $x = \sn t$, $y = \sn(t+\tau)$ we obtain
\begin{multline*}
k^4 \sn^4\tau \ x^4y^2 - k^2 \sn^2\tau \ x^4 - 2k^2 \sn^2\tau \ x^2y^2 + 2k^2 \sn^2\tau \cn\tau \dn\tau \ x^3y \\
+ (1 + k^2 \sn^4\tau) x^2 - 2 \cn\tau \dn\tau \ xy + y^2 - \sn^2 \tau = 0
\end{multline*}
The same equation holds with $x$ and $y$ exchanged. After antisymmetrizing and dividing by $k^2 \sn^2\tau (x^2 - y^2)$ we obtain the equation \eqref{eqn:SnShift}.

In the $\cn$ case the proof is similar.
\end{proof}

A biquadratic equation \eqref{eqn:SnShift} can be viewed as an implicit form of addition formulas for $\sn$ or $\cn$.

\subsubsection{Parametrization theorems}
\label{sec:ParamThm}
We are ready to derive a simultaneous parametrization of all four angles of a quadrilateral with fixed side lengths. With the help of notation \eqref{eqn:abar} equation
\eqref{eqn:AdjX} can be rewritten as
\begin{multline}
\label{eqn:AdjZinS}
\bar a_1 (s - a_1 - a_3) z_1^2 z_2^2 + \bar a_4 (s - a_3 - a_4) z_1^2 + \bar a_2 (s - a_2 - a_3) z_2^2 \\
- 2 a_2 a_4 z_1 z_2 + s \bar a_3 = 0
\end{multline}
Cyclic shifts of indices produce equations for other pairs of adjacent angles.

\begin{prp}
\label{prp:SnEuc}
Let $a = (a_1, a_2, a_3, a_4)$ be such that $a_1 \pm a_2 \pm a_3 \pm a_4 \ne 0$ for all choices of the signs. Define
\[
p_1 = \sqrt{\frac{s(s-a_2-a_3)}{\bar a_1 \bar a_4}} = \sqrt{\frac{a_1 a_4}{\bar a_1 \bar a_4} - 1} \in \R_+ \cup i\R_+
\]
and similarly $p_2$, $p_3$, and $p_4$ through a cyclic shift of indices.
The configuration space $\QorC(a)$ has one of the following parametrizations in terms of the tangents of the half-angles of the quadrilateral.
\begin{enumerate}
\item
If $a_{\min} + a_{\max} < s$ and $a_3 = a_{\min}$, then
\[
z_1 = p_1 \sn t, \quad z_2 = p_2 \sn(t+\tau), \quad z_3 = p_3 \sn(t-K), \quad z_4 = p_4 \sn(t+K+\tau)
\]
where the modulus of $\sn$ is
$$
k = \sqrt{1 - \frac{a_1a_2a_3a_4}{\bar a_1 \bar a_2 \bar a_3 \bar a_4}} \in (0,1),
$$
and the phase shift $\tau$ is determined by
\[
\dn \tau = \sqrt{\frac{a_2a_4}{\bar a_2 \bar a_4}}, \quad \tau \in (0, iK')
\]

\item
If $a_{\min} + a_{\max} > s$ and $a_1 = a_{\max}$, then
\[
z_1 = p_1 \cn t, \quad z_2 = p_2 \cn(t+\tau), \quad z_3 = p_3 \cn(t-K), \quad z_4 = p_4 \cn(t+K+\tau)
\]
where the modulus of $\cn$ is
$$
k = \sqrt{1 - \frac{\bar a_1 \bar a_2 \bar a_3 \bar a_4}{a_1a_2a_3a_4}} \in (0,1),
$$
and the phase shift $\tau$ is determined by
\[
\dn \tau = \sqrt{\frac{\bar a_2 \bar a_4}{a_2a_4}}, \quad \tau \in (0, iK')
\]
\end{enumerate}
\end{prp}
\begin{proof}
Substituting $z_1 = p_1 u_1$ and $z_2 = p_2 u_2$ and using the identities from Lemma \ref{lem:SIdent}, we transform \eqref{eqn:AdjZinS} to
\begin{equation}
\label{eqn:SnPar}
u_1^2 + u_2^2 = \frac{\bar a_2 \bar a_4 - a_2 a_4}{\bar a_2 \bar a_4} u_1^2 u_2^2 - 2 \frac{a_2 a_4}{\sqrt{\bar a_2 \bar a_4}\sqrt{a_2a_4 - \bar a_1 \bar a_3}} u_1 u_2 - \frac{\bar a_1 \bar a_3}{a_2 a_4 - \bar a_1 \bar a_3}
\end{equation}
%

The range of $\dn$ on $(0, iK')$ is $(1, +\infty)$. Due to Lemma \ref{lem:SIdent} we have $a_2a_4 > \bar a_2 \bar a_4$, thus there exists $\tau \in (0, iK')$ with $\dn \tau = \sqrt{\frac{a_2a_4}{\bar a_2 \bar a_4}}$. From this we compute
\[
\sn^2 \tau = - \frac{\bar a_1 \bar a_3}{a_2 a_4 - \bar a_1 \bar a_3}, \quad \cn^2 \tau = \frac{a_2a_4}{a_2 a_4 - \bar a_1 \bar a_3}
\]
Besides, $\cn \tau > 0$ for $\tau \in (0, iK')$. It follows that the coefficients of \eqref{eqn:SnPar} coincide with the coefficients of \eqref{eqn:SnShift} for given values of $k$ and $\tau$. A similar check can be made for other pairs of adjacent angles (where one should be careful with extraction of square roots of negative numbers), which finishes the proof of the $sn$-parametrization.

In the case $a_{\min} + a_{\max} > s$ use the second part of Lemma \ref{lem:SnShift}.
\end{proof}

\subsubsection{The real part of the configuration space}
\label{sec:EllReal}
Under the assumptions of Theorem \ref{prp:SnEuc} we have $p_1, p_2 \in \R_+$, $p_3, p_4 \in i\R_+$. It follows that all of the variables $z_i$ take real values if and only if
\begin{align*}
\Re t \equiv K (\mod 2K), &\text{ for the } \sn\text{-parametrization}\\
\Re t \equiv 0 (\mod 2K), &\text{ for the } \cn\text{-parametrization}
\end{align*}
As in the trigonometric case, we make a variable change so that real values of parameter yield real angles $\phi_i$.

In the $\sn$-case put $t = iu + K$, and in the $\cn$-case put $t = iu$. Jacobi's imaginary transformations \cite[Chapter 2]{AE06} allow to rewrite the parametrizations of Theorem \ref{prp:SnEuc} in terms of elliptic functions with the conjugate modulus.

\begin{prp}
\label{prp:RealEll}
\begin{enumerate}
\item
Let $a_{\min} + a_{\max} < s$, and assume that $a_3 = a_{\min}$. Then the configuration space $\QorC$ can be parametrized as
\begin{gather*}
z_1 = \frac{p_1}{\dn(u;k')}, \quad z_2 = \frac{p_2}{\dn(u+\sigma;k')}\\
z_3 = \frac{p_3}{\dn(u+iK;k')} = ip_3 \frac{\sn(u;k')}{\cn(u;k'}\\
z_4 = \frac{p_4}{\dn(u-iK+\sigma;k')} = -ip_4 \frac{\sn(u+\sigma;k')}{\cn(u+\sigma;k')}
\end{gather*}
where
$$
k' = \sqrt{\frac{a_1a_2a_3a_4}{\bar a_1 \bar a_2 \bar a_3 \bar a_4}} \in (0,1),
$$
the amplitudes $p_i$ are given by
\[
p_1 = \sqrt{\frac{s(s-a_2-a_3)}{\bar a_1 \bar a_4}} = \sqrt{\frac{a_1 a_4}{\bar a_1 \bar a_4} - 1}, \quad \text{etc.}
\]
and the phase shift $\sigma$ is determined by
\[
\sn(\sigma;k') = \sqrt{\frac{\bar a_1 \bar a_3}{a_2 a_4}}, \quad \sigma \in (0, K')
\]
\item
Let $a_{\min} + a_{\max} > s$, and assume that $a_1 = a_{\max}$. Then the configuration space $\QorC$ can be parametrized as
\begin{gather*}
z_1 = \frac{p_1}{\cn(u;k')}, \quad z_2 = \frac{p_2}{\cn(u+\sigma;k')}\\
z_3 = \frac{p_3}{\cn(u+iK;k')} = ik'p_3 \frac{\sn(u;k')}{\dn(u;k')}\\
z_4 = \frac{p_4}{\cn(u-iK+\sigma)} = -ik'p_4 \frac{\sn(u+\sigma;k')}{\dn(u+\sigma;k')}
\end{gather*}
where
$$
k' = \sqrt{\frac{\bar a_1 \bar a_2 \bar a_3 \bar a_4}{a_1a_2a_3a_4}} \in (0,1),
$$
the amplitudes $p_i$ are given by
\[
p_1 = \sqrt{\frac{s(s-a_2-a_3)}{\bar a_1 \bar a_4}} = \sqrt{\frac{a_1 a_4}{\bar a_1 \bar a_4} - 1}, \quad \text{etc.}
\]
and the phase shift $\sigma$ is determined by
\[
\sn(\sigma;k') = \sqrt{\frac{a_1 a_3}{\bar a_2 \bar a_4}}, \quad \sigma \in (0, K')
\]
\end{enumerate}
\end{prp}

\begin{figure}[ht]
\centering
\begin{picture}(0,0)%
\includegraphics{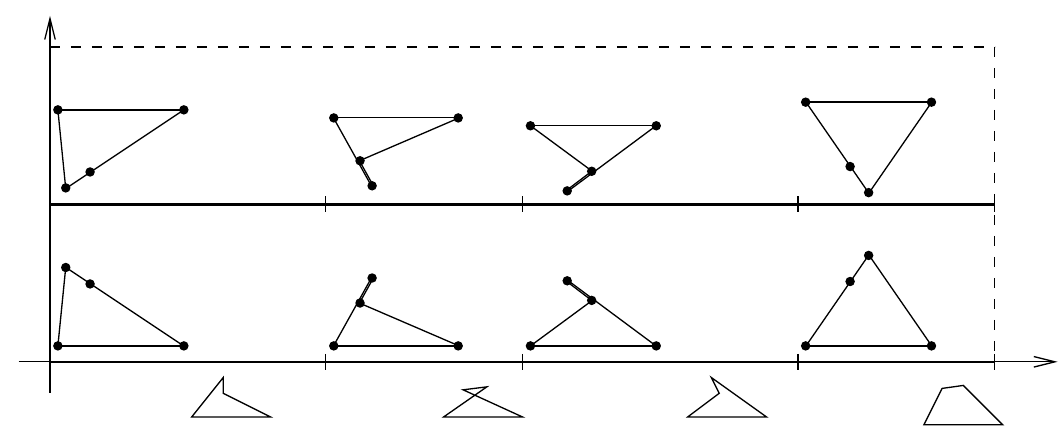}%
\end{picture}%
\setlength{\unitlength}{3315sp}%
\begingroup\makeatletter\ifx\SetFigFont\undefined%
\gdef\SetFigFont#1#2#3#4#5{%
  \reset@font\fontsize{#1}{#2pt}%
  \fontfamily{#3}\fontseries{#4}\fontshape{#5}%
  \selectfont}%
\fi\endgroup%
\begin{picture}(6057,2430)(-284,-1333)
\put( 46,974){\makebox(0,0)[lb]{\smash{{\SetFigFont{7}{8.4}{\rmdefault}{\mddefault}{\updefault}{\color[rgb]{0,0,0}$\Im$}%
}}}}
\put(1576,-1096){\makebox(0,0)[lb]{\smash{{\SetFigFont{7}{8.4}{\rmdefault}{\mddefault}{\updefault}{\color[rgb]{0,0,0}$K'-\sigma$}%
}}}}
\put(5671,-826){\makebox(0,0)[lb]{\smash{{\SetFigFont{7}{8.4}{\rmdefault}{\mddefault}{\updefault}{\color[rgb]{0,0,0}$\Re$}%
}}}}
\put(-269,794){\makebox(0,0)[lb]{\smash{{\SetFigFont{7}{8.4}{\rmdefault}{\mddefault}{\updefault}{\color[rgb]{0,0,0}$4iK$}%
}}}}
\put(-269,-106){\makebox(0,0)[lb]{\smash{{\SetFigFont{7}{8.4}{\rmdefault}{\mddefault}{\updefault}{\color[rgb]{0,0,0}$2iK$}%
}}}}
\put( 46,-1096){\makebox(0,0)[lb]{\smash{{\SetFigFont{7}{8.4}{\rmdefault}{\mddefault}{\updefault}{\color[rgb]{0,0,0}$0$}%
}}}}
\put(2701,-1096){\makebox(0,0)[lb]{\smash{{\SetFigFont{7}{8.4}{\rmdefault}{\mddefault}{\updefault}{\color[rgb]{0,0,0}$K'$}%
}}}}
\put(4276,-1096){\makebox(0,0)[lb]{\smash{{\SetFigFont{7}{8.4}{\rmdefault}{\mddefault}{\updefault}{\color[rgb]{0,0,0}$2K'-\sigma$}%
}}}}
\put(5401,-1096){\makebox(0,0)[lb]{\smash{{\SetFigFont{7}{8.4}{\rmdefault}{\mddefault}{\updefault}{\color[rgb]{0,0,0}$2K'$}%
}}}}
\end{picture}%
\caption{The real part of the configuration space in the case $a_{\min} + a_{\max} < s$.}
\label{fig:RealSn}
\end{figure}

\begin{figure}[ht]
\centering
\begin{picture}(0,0)%
\includegraphics{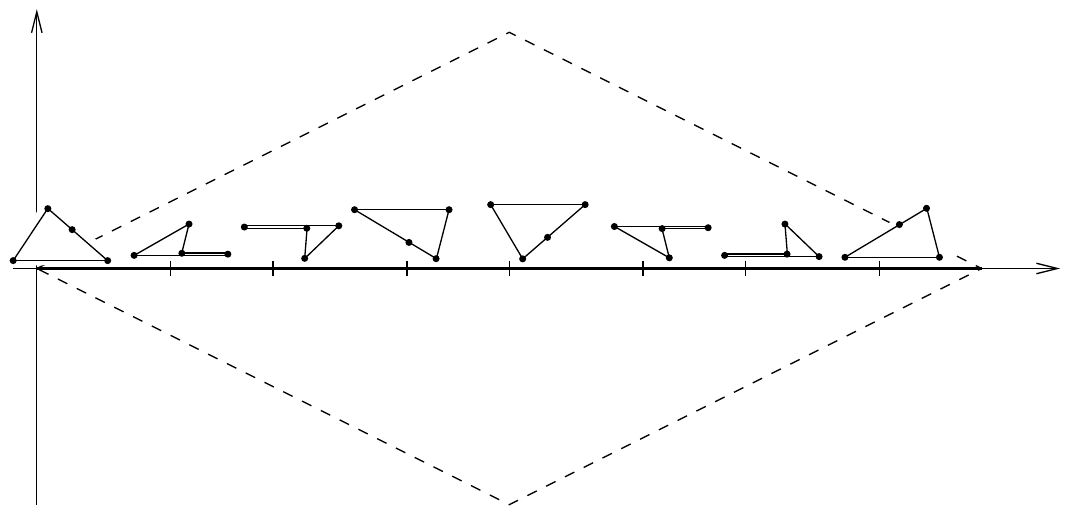}%
\end{picture}%
\setlength{\unitlength}{3315sp}%
\begingroup\makeatletter\ifx\SetFigFont\undefined%
\gdef\SetFigFont#1#2#3#4#5{%
  \reset@font\fontsize{#1}{#2pt}%
  \fontfamily{#3}\fontseries{#4}\fontshape{#5}%
  \selectfont}%
\fi\endgroup%
\begin{picture}(6073,2880)(-210,-1873)
\put(-89,-646){\makebox(0,0)[lb]{\smash{{\SetFigFont{7}{8.4}{\rmdefault}{\mddefault}{\updefault}{\color[rgb]{0,0,0}$0$}%
}}}}
\put(5356,-646){\makebox(0,0)[lb]{\smash{{\SetFigFont{7}{8.4}{\rmdefault}{\mddefault}{\updefault}{\color[rgb]{0,0,0}$4K'$}%
}}}}
\put(541,-646){\makebox(0,0)[lb]{\smash{{\SetFigFont{7}{8.4}{\rmdefault}{\mddefault}{\updefault}{\color[rgb]{0,0,0}$K'-\sigma$}%
}}}}
\put(1351,-646){\makebox(0,0)[lb]{\smash{{\SetFigFont{7}{8.4}{\rmdefault}{\mddefault}{\updefault}{\color[rgb]{0,0,0}$K'$}%
}}}}
\put(1846,-646){\makebox(0,0)[lb]{\smash{{\SetFigFont{7}{8.4}{\rmdefault}{\mddefault}{\updefault}{\color[rgb]{0,0,0}$2K'-\sigma$}%
}}}}
\put(2611,-646){\makebox(0,0)[lb]{\smash{{\SetFigFont{7}{8.4}{\rmdefault}{\mddefault}{\updefault}{\color[rgb]{0,0,0}$2K'$}%
}}}}
\put(4546,-646){\makebox(0,0)[lb]{\smash{{\SetFigFont{7}{8.4}{\rmdefault}{\mddefault}{\updefault}{\color[rgb]{0,0,0}$4K' - \sigma$}%
}}}}
\put(3196,-646){\makebox(0,0)[lb]{\smash{{\SetFigFont{7}{8.4}{\rmdefault}{\mddefault}{\updefault}{\color[rgb]{0,0,0}$3K'-\sigma$}%
}}}}
\put(3961,-646){\makebox(0,0)[lb]{\smash{{\SetFigFont{7}{8.4}{\rmdefault}{\mddefault}{\updefault}{\color[rgb]{0,0,0}$3K'$}%
}}}}
\put(5761,-466){\makebox(0,0)[lb]{\smash{{\SetFigFont{7}{8.4}{\rmdefault}{\mddefault}{\updefault}{\color[rgb]{0,0,0}$\Re$}%
}}}}
\put( 46,884){\makebox(0,0)[lb]{\smash{{\SetFigFont{7}{8.4}{\rmdefault}{\mddefault}{\updefault}{\color[rgb]{0,0,0}$\Im$}%
}}}}
\put(2746,839){\makebox(0,0)[lb]{\smash{{\SetFigFont{7}{8.4}{\rmdefault}{\mddefault}{\updefault}{\color[rgb]{0,0,0}$2K'+2iK$}%
}}}}
\end{picture}%
\caption{The real part of the configuration space in the case $a_{\min} + a_{\max} > s$.}
\label{fig:RealCn}
\end{figure}

The shapes of quadrilaterals corresponding to special values of parameter $t$ are shown on Figures \ref{fig:RealSn} and \ref{fig:RealCn}. Note that the real part consists of one component (a quadrilateral can be transformed to its mirror image by a continuous deformation) if $a_{\min} + a_{\max} > s$, and of two components (no continuous deformation between a quadrilateral and its mirror image) if $a_{\min} + a_{\max} < s$.

Since $z_i^{-1} = \cot\frac{\phi_i}2$, Proposition \ref{prp:RealEll} implies Theorem \ref{thm:Param}.

\section{The configuration spaces of non-oriented quadrilaterals}
\label{sec:NonOr}
\subsection{Double covers}
\label{sec:DoubCov}
For $a = (a_1, a_2, a_3, a_4)$ satisfying the conditions \eqref{eqn:QuadIneqA} and \eqref{eqn:QuadIneqB}, let $Q(a)$ denote the space of congruence classes of quadrilaterals in $\R^2$ with side lengths $a$ in this cyclic order. This time, orientation-reversing congruences are allowed, so that to almost every element of $Q(a)$ there correspond two elements of $\Qor(a)$. Exceptions are quadrilaterals with collinear vertices.

If $a$ is elliptic, that is $a_1 \pm a_2 \pm a_3 \pm a_4 \ne 0$ for every choice of signs, then there are no exceptional quadrilaterals and the natural map $\Qor(a) \to Q(a)$ is a double cover (disconnencted for non-Grashof quadrilaterals). The equivalence relation on $\Qor(a)$ looks very simple in terms of the angles of a quadrilateral:
\[
(\phi_1, \phi_2, \phi_3, \phi_4) \sim (-\phi_1, -\phi_2, -\phi_3, -\phi_4)
\]
and has the same form in the variables $z_i = \tan \frac{\phi}2$. By extending this to $\QorC(a)$ we define the complexified configuration space of non-oriented quadrilaterals
\[
Q^{\C}(a) := \QorC(a)/(z_1, z_2, z_3, z_4) \sim (-z_1, -z_2, -z_3, -z_4)
\]
The quotient map
\[
\QorC(a) \to Q^{\C}(a)
\]
is a double cover between elliptic curves.
In terms of the holomorphic parameter $t$ of Proposition \ref{prp:SnEuc} we have $t \sim t+2K$ both in the $\sn$ and in the $\cn$ case. This identifies $Q(a)$ with $\C/2K\Z + 2K'iZ$ and proves the first part of Theorem \ref{thm:NonOr}.

\subsection{Diagonal lengths as coordinates on the configuration space}
\label{sec:DiagLengths}
The cosines $\cos \phi_i$ are meromorphic functions on the elliptic curve $Q^{\C}(a)$. On the other hand, by the cosine law they are linearly related to the squares of diagonal lengths. Since a quadrilateral with fixed edge lengths is uniquely determined up to congruence by the lengths of its diagonals, this provides a natural embedding of $Q^{\C}(a)$ in $(\CP^1)^2$.

\begin{lem}
Let $x$ and $y$ be the lengths of diagonals ($x$ is separating $a_1$ and $a_2$ from $a_3$ and $a_4$), and
\[
u = x^2, \quad v = y^2
\]
Then the configuration space $Q(a)$ is an algebraic curve given by the equation
\begin{equation}
\label{eqn:DiagCoord}
u^2 v + u v^2 + 2d_{11}uv + d_{10}u + d_{01}v + d_{00} = 0, \quad \text{where}
\end{equation}
\begin{equation*}
\begin{aligned}
d_{11} &= -\frac12(a_1^2 + a_2^2 + a_3^2 + a_4^2)\\
d_{10} &= (a_1^2 - a_4^2)(a_2^2 - a_3^2)\\
d_{01} &= (a_1^2 - a_2^2)(a_4^2 - a_3^2)\\
d_{00} &= (a_1^2 - a_2^2 + a_3^2 - a_4^2)(a_1^2 a_3^2 - a_2^2 a_4^2)
\end{aligned}
\end{equation*}
\end{lem}
\begin{proof}
Viewing a planar quadrilateral as a tetrahedron of zero volume, we equate to zero the Gram determinant of the vectors (two sides and a diagonal) issued from one of its vertices.
\[
\det
\begin{pmatrix}
a_1^2 & \frac{a_1^2 + x^2 - a_2^2}2 & \frac{a_1^2 + a_4^2 - y^2}2\\
\frac{a_1^2 + x^2 - a_2^2}2 & x^2 & \frac{x^2 + a_4^2 - a_3^2}2\\
\frac{a_1^2 + a_4^2 - y^2}2 & \frac{x^2 + a_4^2 - a_3^2}2 & a_4^2
\end{pmatrix}
= 0
\]
A simple computation leads to the equation in the lemma.
\end{proof}

\subsection{Conjugate quadrilaterals}
Recall the following notations.
\begin{dfn}
For $a \in \R^4$ the \emph{conjugate quadruple} $\bar a \in \R^4$ is defined as
$$
\begin{aligned}
\bar a_1 &:= s-a_1 = \frac{-a_1 + a_2 + a_3 + a_4}2, \quad &\bar a_2 &:= s-a_2 = \frac{a_1 - a_2 + a_3 + a_4}2,\\
\bar a_3 &:= s-a_3 = \frac{a_1 + a_2 - a_3 + a_4}2, \quad &\bar a_4 &:= s-a_4 = \frac{a_1 + a_2 + a_3 - a_4}2
\end{aligned}
$$
\end{dfn}
\begin{lem}
The map $a \mapsto \bar a$ is an involution on the space of quadruples $a \in \R^4$ satisfying inequalities \eqref{eqn:QuadIneqA} and \eqref{eqn:QuadIneqB}.
\end{lem}
\begin{proof}
We have $\bar{\bar a} = a$ since $\sum \bar a_i = \sum a_i$. The quadrangle inequality $s > a_i$ is equivalent to the positivity of the length $\bar a_i > 0$.
\end{proof}
Thus a quadruple conjugate to the edge lengths of a quadrilateral can itself be used as edge lengths of another quadrilateral.

Note that $\bar a = a$ if and only if all $a_i$ are equal, and $\bar a$ coincides with $a$ up to the dihedral group action on $\{1, 2, 3, 4\}$ if and only if $a_{\min} + a_{\max} = s$. Thus for $a_1 \pm a_2 \pm a_3 \pm a_4 \ne 0$ quadrilaterals in $Q(\bar a)$ are not congruent to quadrilaterals in $Q(a)$, even if we ignore the marking of sides.

The more surprising is the second part of Theorem \ref{thm:NonOr}: for every quadrilateral with side lengths $a$ and diagonal lengths $x$ and $y$ there is a quadrilateral with side lengths $\bar a$ and diagonal lengths $x$ and $y$.
For this, it suffices to show that the coefficients $d_{ij}$ in \eqref{eqn:DiagCoord} don't change if we replace $a$ by $\bar a$. This follows from $a+b = \bar c + \bar d$, $a-b = \bar b - \bar a$, which implies that the coefficients $d_{01}$ and $d_{10}$ don't change and from Lemma \ref{lem:SIdent2} below, which takes care of $d_{11}$ and $d_{00}$.

\begin{lem}
\label{lem:SIdent2}
Let
\[
s = \frac{a+b+c+d}2, \quad \bar a = s-a, \quad \ldots
\]
Then we have
\begin{gather*}
a^2 + b^2 + c^2 + d^2 = \bar a^2 + \bar b^2 + \bar c^2 + \bar d^2\\
ab + cd = \bar a \bar b + \bar c \bar d\\
ab - cd = \frac12 (\bar c^2 + \bar d^2 - \bar a^2 - \bar b^2)
\end{gather*}
\end{lem}

It is easy to see that
\[
a_{\min} + a_{\max} < s \Leftrightarrow \bar a_{\min} + \bar a_{\max} > \bar s = s
\]
Thus $\QorC(a)$ is a quotient of $\C$ by a rectangular lattice if and only if $\QorC(\bar a)$ is a quotient of $\C$ by a rhombic lattice. Thus the orientation forgetting covers over $Q^{\C}(a) \cong Q^{\C}(\bar a)$ look as shown on Figure \ref{fig:Covers}, and Theorem \ref{thm:NonOr} is proved.


\section{Periodic quadrilaterals}
\label{sec:Periodic}
\subsection{Folding of conic quadrilaterals}
\begin{prp}
For a quadrilateral of conic type, the composition of foldings
\[
F_3 \circ F_4 \colon \QorC(a) \to \QorC(a)
\]
has a unique fixed point, namely the ``flattened'' quadrilateral with collinear vertices.

For any initial position $q \in \QorC(a)$, the sequence $q_n = (F_3 \circ F_4)^n(q)$ converges to the flattened quadrilateral, with the angles of $q_n$ tending to $0$ or $\pi$ exponentially.
\end{prp}
\begin{proof}
Proposition \ref{prp:Param1+3=2+4} implies that in terms of the parameter $t$ the folding maps take the form
\[
F_3(t) = \pi - t, \quad F_4(t) = \pi - 2i\sigma - t
\]
so that $(F_3 \circ F_4)(t) = t + 2i\sigma$. Thus
\[
(F_3 \circ F_4)^n(t) \to \infty
\]
and $\infty$ is a unique fixed point of $F_4 \circ F_3$. The exponential decay of the angles follows from
\[
z_1((F_3 \circ F_4)^n(t)) = p_1 \sin(t + 2ni\sigma) \sim e^{2n\sigma}
\]
and $z_1 = \tan \frac{\phi_1}2 \sim \frac2{\pi - \phi_1}$ as $\phi_1 \to \pi$.
\end{proof}

\subsection{Folding of elliptic quadrilaterals}
\begin{proof}[Proof of the Darboux porism]
Proposition \ref{prp:SnEuc} implies that the folding maps $F_i \colon \QorC(a) \to \QorC(a)$ are involutions of the form $t \mapsto t_i - t$. Thus the composition $F_3 \circ F_4$ is a translation $t \mapsto t + (t_3-t_4)$. The $n$-th iteration $(F_3 \circ F_4)^n$ is a translation by $n(t_3-t_4)$. If it has a fixed point, then $n(t_3-t_4)$ belongs to the lattice of periods, and consequently $(F_3 \circ F_4)^n = \id$.
\end{proof}

The definition of $n$-periodicity given in the Introduction depends on the choice of an ordered pair of adjacent vertices. Let us clarify this dependence.

\begin{lem}
If the map $F_3 \circ F_4 \colon \QorC(a) \to \QorC(a)$ has order $n$, then the maps $F_4 \circ F_3$, $F_1 \circ F_2$, and $F_2 \circ F_1$ also have order $n$. The map $F_2 \circ F_3$ also has a finite order, but possibly different from $n$.
\end{lem}
\begin{proof}
The maps $F_4 \circ F_3$ and $F_3 \circ F_4$ have the same order because they are conjugate to each other. Proposition \ref{prp:SnEuc} implies that the maps $F_3 \circ F_4$ and $F_1 \circ F_2$ are translations by the same vector (modulo the lattice of periods), therefore have the same order. Finally, the same theorem shows that
\[
F_3 \circ F_4(t) = t + 2\tau, \quad F_2 \circ F_3(t) = t - 2K - 2\tau
\]
Thus one translation is commensurable with the lattice of periods if and only if the other is.
\end{proof}

\begin{dfn}
A quadrilateral is called \emph{$n$-periodic}, if both $F_3 \circ F_4$ and $F_2 \circ F_3$ have order $n$. It is called \emph{$(n,m)$-periodic}, if one of the maps $F_3 \circ F_4$ and $F_2 \circ F_3$ has order $n$, and the other has order $m$.
\end{dfn}

On the other hand, the folding period \emph{modulo orientation}, that is the order of the natural map $\hat F_3 \circ \hat F_4 \colon Q^{\C}(a) \to Q^{\C}(a)$ does not depend on the choice of a pair of adjacent vertices because $\hat F_1 = \hat F_3$ and $\hat F_2 = \hat F_4$.


\begin{prp}
\label{prp:Period}
Let $\sigma \in (0,K')$ be the basic shift as defined in Theorem \ref{prp:RealEll}. Assume that $\sigma = \frac{m}{n}K'$ with $m$ and $n$ coprime. Then all quadrilaterals with side lengths $(a_1,a_2,a_3,a_4)$ are $n$-periodic up to symmetry. Besides,
\begin{enumerate}
\item in the $\sn$-case
\[
\begin{aligned}
n \text{ is even} &\Rightarrow \text{the period length is }n\\
n \text{ is odd} &\Rightarrow \text{the period lengths are }(n,2n)
\end{aligned}
\]
\item in the $\cn$-case
\[
\begin{aligned}
n \text{ is even} &\Rightarrow \text{the period length is }2n\\
n \text{ is odd} &\Rightarrow \text{the period lengths are }(n,2n)
\end{aligned}
\]
\end{enumerate}
\end{prp}
\begin{proof}
Let us use the parameter change made in Section \ref{sec:EllReal}, see Figures \ref{fig:RealSn} and \ref{fig:RealCn} for illustration. In terms of the real parameter $u$ the two shifts performed by different pairs of adjacent foldings are
\[
2\sigma \text{ and } 2iK - 2\sigma
\]
The $n$-fold iteration results in the shifts by
\[
n \cdot 2\sigma = 2mK', \quad n(2iK - 2\sigma) = 2niK - 2mK'
\]
These are the smallest multiples of $2\sigma$ that belong to the lattice $2K'\Z + 2iK\Z$, which is the period lattice for $Q^{\C}(a)$, see Section \ref{sec:DoubCov}.

More details are needed to determine the period on $\QorC(a)$.
If $n$ is even, then $m$ is odd. Thus both of the above shifts are periods of $\dn$ and half-periods of $\cn$.
If $n$ is odd, then the first shift is a period of $\dn$, but the second only a half-period. If $m$ is also odd, then the first one is a half-period for $\cn$, and the second a period; vice versa if $m$ is even.
\end{proof}

\subsection{Quadrilaterals with small periods}
\subsubsection{Period $2$: orthodiagonal quadrilaterals}
The composition of foldings $F_3 \circ F_4$ has order $2$ if and only if $F_3$ and $F_4$ commute. This happens if and only if the third and the fourth vertices stay on their respective diagonals after the folding. This, in turn, is equivalent to diagonals being orthogonal. Pythagorean theorem implies that diagonals of a quadrilateral are orthogonal if and only if $a_1^2 + a_3^2 = a_2^2 + a_4^2$. This proves the period $2$ case of the Darboux porism without using elliptic functions.

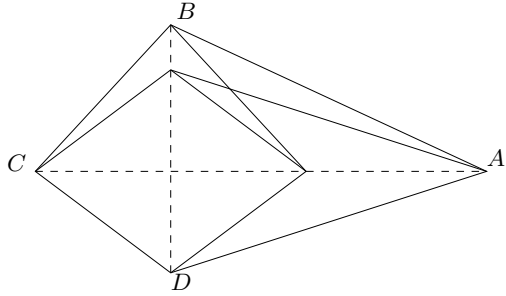
\begin{figure}[ht]
\centering
\begin{picture}(0,0)%
\includegraphics{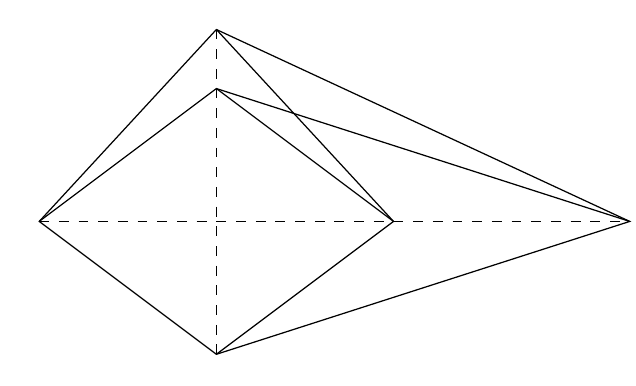}%
\end{picture}%
\setlength{\unitlength}{3108sp}%
\begingroup\makeatletter\ifx\SetFigFont\undefined%
\gdef\SetFigFont#1#2#3#4#5{%
  \reset@font\fontsize{#1}{#2pt}%
  \fontfamily{#3}\fontseries{#4}\fontshape{#5}%
  \selectfont}%
\fi\endgroup%
\begin{picture}(3855,2338)(-239,-1331)
\put(3601,-286){\makebox(0,0)[lb]{\smash{{\SetFigFont{9}{10.8}{\rmdefault}{\mddefault}{\updefault}{\color[rgb]{0,0,0}$A$}%
}}}}
\put(1126,884){\makebox(0,0)[lb]{\smash{{\SetFigFont{9}{10.8}{\rmdefault}{\mddefault}{\updefault}{\color[rgb]{0,0,0}$B$}%
}}}}
\put(-224,-331){\makebox(0,0)[lb]{\smash{{\SetFigFont{9}{10.8}{\rmdefault}{\mddefault}{\updefault}{\color[rgb]{0,0,0}$C$}%
}}}}
\put(1081,-1276){\makebox(0,0)[lb]{\smash{{\SetFigFont{9}{10.8}{\rmdefault}{\mddefault}{\updefault}{\color[rgb]{0,0,0}$D$}%
}}}}
\end{picture}%
\caption{Folding of an orthodiagonal quadrilateral.}
\end{figure}

\begin{prp}
\label{prp:Period2}
A quadrilateral is $2$-periodic if and only if its diagonal are orthogonal to each other. This property depends only on the side lengths of the quadrilateral, and is equivalent to
\begin{equation}
a_1^2 + a_3^2 = a_2^2 + a_4^2
\end{equation}
\end{prp}

Note that in the orthodiagonal case equation \eqref{eqn:AdjX} can be rewritten as
\[
(A_1 z_1 + B_1 z_1^{-1})(A_2 z_2 + B_2z_2^{-1}) = 1
\]
which also makes the commutation of foldings apparent.

\subsubsection{Period $4$}
Introduce the notation
\[
A = a_1a_3, \quad B = a_2 a_4, \quad \bar A = \bar a_1 \bar a_3, \quad \bar B = \bar a_2 \bar a_4
\]
Then by Theorem \ref{prp:RealEll} we have
\begin{subequations}
\begin{equation}
\label{eqn:DataSn}
(k')^2 = \frac{AB}{\bar A \bar B}, \quad \sn^2 \sigma = \frac{\bar A}B, \quad \cn^2\sigma = 1 - \frac{\bar A}B, \quad \dn^2 \sigma = 1 - \frac{A}{\bar B}
\end{equation}
\begin{equation}
\label{eqn:DataCn}
(k')^2 = \frac{\bar A \bar B}{AB}, \quad \sn^2 \sigma = \frac{A}{\bar B}, \quad \cn^2\sigma = 1 - \frac{A}{\bar B}, \quad \dn^2 \sigma = 1 - \frac{\bar A}{B}
\end{equation}
\end{subequations}
in the $\sn$- and in the $\cn$-case, respectively. (All elliptic functions have elliptic parameter $k'$.) Note that the numbers $A, B, \bar A, \bar B$ satisfy the identities
\[
A+B = \bar A + \bar B, \quad A\bar A - B \bar B = (A+B)(A-\bar B)
\]
Also we have $a_{\min} + a_{\max} < s \Leftrightarrow AB < \bar A \bar B$.

From \eqref{eqn:DataSn} we compute
\[
\sn 2\sigma = \frac{2\sn\sigma\cn\sigma\dn\sigma}{1-k^2\sn^4\sigma} = \frac{2\sqrt{\bar A \bar B}}{\bar A + \bar B}
\]
This gives another proof of Proposition \ref{prp:Period2}, since $\sigma \in (0,K')$ equals $\frac{K'}2$ if and only if $\sn(2\sigma;k') = 1$, which is equivalent to
\[
2 \sqrt{\bar A \bar B} = \bar A + \bar B
\]
that is, to $\bar A = \bar B$, which by Lemma \ref{lem:SIdent2} is equivalent to $a_1^2 + a_3^2 = a_2^2 + a_4^2$.


\begin{prp}
A quadrilateral is $4$-periodic if and only if one of the following takes place:
\begin{enumerate}
\item
$a_{\min} + a_{\max} > s$ and $a_1a_3 = a_2a_4$;
\item
$a_{\min} + a_{\max} < s$ and
\[
(a_1a_3 - a_2a_4)(a_1a_3 + a_2a_4) = \pm 2 \sqrt{\bar a_1 \bar a_2 \bar a_3 \bar a_4}(a_1^2 + a_3^2 - a_2^2 - a_4^2)
\]
\end{enumerate}
\end{prp}
\begin{proof}
According to Proposition \ref{prp:Period}, a quadrilateral is $4$-periodic if either $a_{\min} + a_{\max} < s$ and $\sigma \in \{\frac{K'}4, \frac{3K'}4\}$ or $a_{\min} + a_{\max} > s$ and $\sigma = \frac{K'}2$.

In the latter case by symmetry with the $\sn$-case we have
\[
\sigma = \frac{K'}2 \Leftrightarrow \frac{2\sqrt{AB}}{A + B} = 1 \Leftrightarrow A = B
\]

In the former case we compute
\[
\sn 2\sigma = \frac{2\sqrt{\bar A \bar B}}{\bar A + \bar B} \Rightarrow \cn 2\sigma = \pm \frac{|\bar A - \bar B|}{A + B}, \quad \dn 2\sigma = \frac{|A-B|}{A+B}
\]
where $\cn 2\sigma > 0$ if and only if $\sigma \in (0, \frac{K'}2)$. Using once again the formula for $\sn 2x$, we obtain
\[
\sn 4\sigma = \pm 4 \frac{\sqrt{\bar A \bar B} |\bar A - \bar B|}{(A+B)|A-B|}
\]
Since $\sigma \in \{\frac{K'}4, \frac{3K'}4\} \Leftrightarrow \sn 4\sigma = \pm 1$, we obtain the second condition.
\end{proof}

\subsubsection{Period 3}
\begin{prp}
\label{prp:Period3}
A quadrilateral is $(3,6)$-periodic if and only if
\[
(a_1a_3 + a_2a_4)^2 \in \{4a_1a_3\bar a_2 \bar a_4, 4a_2a_4\bar a_1 \bar a_3, 4a_1a_3\bar a_1 \bar a_3, 4 a_2a_4\bar a_2\bar a_4\}
\]
\end{prp}
\begin{proof}
Due to Proposition \ref{prp:Period} we have to find the necessary and sufficient conditions for $\sigma \in \{\frac{K'}3, \frac{2K'}3\}$.
For $\sigma \in (0,K')$ we have
\begin{gather*}
\sigma = \frac{K'}3 \Leftrightarrow \sn 2\sigma = \sn(K'-\sigma) = \frac{\cn \sigma}{\dn \sigma}\\
\sigma = \frac{2K'}3 \Leftrightarrow \sn 2\sigma = \sn(2K'-\sigma) = \sn\sigma
\end{gather*}
This leads to the equations $(A+B)^2 = 4\bar A B$ and $(A+B)^2 = 4B\bar B$ in the $\sn$-case, and to $(A+B)^2 = 4A \bar B$ and $(A+B)^2 = 4 A \bar A$ in the $\cn$-case. Since we were making assumptions $a_3 = a_{\min}$ and $a_1 = a_{\max}$ in the $\sn$- and $\cn$-cases respectively, omitting them leads to symmetrizing with respect to $A$ and $B$. Thus all four conditions are valid both in the $\sn$- and in the $\cn$-case.
\end{proof}

\begin{exl}
For $a_1=a_2=a_3=a$ and $a_4=b$ we have $\bar a_1 = \bar a_2 = \bar a_3 = \frac{a+b}2$ and $\bar a_4 = \frac{3a-b}2$, so that
\[
(a_1a_3+a_2a_4)^2 = a^2(a+b)^2 = 4a_1a_3\bar a_1 \bar a_3
\]
and hence every quadrilateral with side lengths $(a,a,a,b)$ is $(3,6)$-periodic. For a geometric proof of this fact, see \cite[Figure 7]{Zvo97}. Note that the quadrilateral is Grashof for $a>b$ and non-Grashof for $a<b$.
\end{exl}

%
%

\subsection{The general periodicity condition}
We have $\sigma = \frac{m}{n} K'$ (and thus the quadrilateral is periodic) if and only if
\[
\cn(n\sigma; k') \in \{0, -1, 1\}
\]
It is possible to express $\cn(n\sigma)$ as a rational function of $\cn\sigma$, so that $\sigma = \frac{m}{n}K'$ becomes an algebraic conditions on the side lengths. These rational functions can be computed recursively, with the help of the following formulas that generalize the Chebyshev polynomials.
\[
R_0(t) = 1, \quad R_1(t) = t, \quad R_{n+1}(t) = \frac{2t R_n(t)}{1-(k')^2(1-t^2)(1-R_n^2(t))} - R_{n-1}(t)
\]

Benoist and Hulin \cite{BH04} relate the dynamics of quadrilateral folding to the Poncelet dynamics on the circles of radii $R = a_2a_4$ and $r = \frac12(a_4^2 - a_3^2 + a_2^2 - a_1^2)$ whose centers are at the distance $a = a_1a_3$ from each other. From the known conditions on short trajectories they obtain characterization of quadrilaterals with small folding period. One may then use Cayley's explicit conditions on the periodicity of Poncelet trajectories to obtain the periodicity condition for quadrilateral folding. We choose a more direct approach, following the general principle derived by Griffiths and Harris \cite{GH78} that also leads to a modern proof of Cayley's conditions. This results in our Theorem \ref{thm:Period}, which we prove below.

\begin{thm}[Griffiths--Harris \cite{GH78}]
\label{thm:GH}
Let $E$ be an elliptic curve and $p, q \in E$. For an integer $n \ge 3$ we have $n(p-q) = 0$ with respect to the additive structure on $E$ if and only if the following condition is satisfied.

Choose a meromorphic function $x$ of order $2$ on $E$ such that $x$ has a double pole at $p$ and a zero at $q$. Let $a, b, c \in \C$ be the values of $x$ at its branch points other than $p$, so that $E$ is isomorphic to the Riemann surface of the function
\[
y = \sqrt{(x-a)(x-b)(x-c)} = \sum_{i=0}^\infty A_i x^i
\]
The condition on $p-q$ being of order $n$ is

\begin{align*}
&\begin{vmatrix}
A_2 & A_3 & \ldots & A_{k+1}\\
A_3 & A_4 & \ldots & A_{k+2}\\
\vdots & \vdots & \ddots & \vdots\\
A_{k+1} & A_{k+2} & \ldots & A_{2k}
\end{vmatrix}
= 0, \quad \text{ if } n = 2k+1\\
&\begin{vmatrix}
A_3 & A_4 & \ldots & A_{k+1}\\
A_4 & A_5 & \ldots & A_{k+2}\\
\vdots & \vdots & \ddots & \vdots\\
A_{k+1} & A_{k+2} & \ldots & A_{2k-1}
\end{vmatrix}
= 0, \quad \text{ if } n = 2k
\end{align*}
\end{thm}

Our first goal is to construct a function $x$ as described in the Griffiths-Harris theorem.
Consider the function $u$ equal to the squared length of the diagonal separating the sides $a_1, a_2$ from $a_3, a_4$. At the branch points, $u$ takes the values
\[
(a_1 + a_2)^2, \quad (a_1 - a_2)^2, \quad (a_3 + a_4)^2, \quad (a_3 - a_4)^2
\]
Let $p \in Q(a)$ be such that $u(p) = (a_3 + a_4)^2$. If $a_1+a_2 > a_3+a_4$, then $p$ lies in the real part of $Q(a)$ and corresponds to the quadrilateral in which the sides $a_3$ and $a_4$ are aligned to form a segment of length $a_3 + a_4$. The quadrilateral $p$ is a fixed point for the folding $\hat F_4$, and the folding $\hat F_3$ maps it to the quadrilateral on Figure \ref{fig:pq}, right.

\begin{figure}
\centering
\begin{picture}(0,0)%
\includegraphics{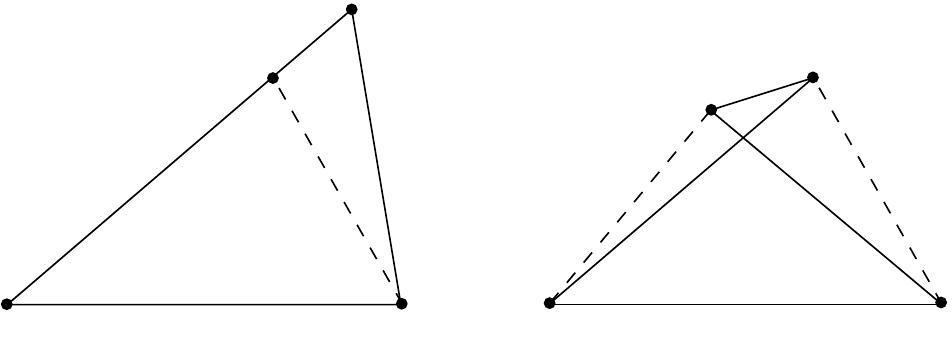}%
\end{picture}%
\setlength{\unitlength}{4144sp}%
\begingroup\makeatletter\ifx\SetFigFont\undefined%
\gdef\SetFigFont#1#2#3#4#5{%
  \reset@font\fontsize{#1}{#2pt}%
  \fontfamily{#3}\fontseries{#4}\fontshape{#5}%
  \selectfont}%
\fi\endgroup%
\begin{picture}(4334,1528)(328,-658)
\put(2088, 71){\makebox(0,0)[lb]{\smash{{\SetFigFont{9}{10.8}{\rmdefault}{\mddefault}{\updefault}{\color[rgb]{0,0,0}$a_2$}%
}}}}
\put(739, 31){\makebox(0,0)[lb]{\smash{{\SetFigFont{9}{10.8}{\rmdefault}{\mddefault}{\updefault}{\color[rgb]{0,0,0}$a_4$}%
}}}}
\put(3669,-598){\makebox(0,0)[lb]{\smash{{\SetFigFont{9}{10.8}{\rmdefault}{\mddefault}{\updefault}{\color[rgb]{0,0,0}$a_1$}%
}}}}
\put(3943,-97){\makebox(0,0)[lb]{\smash{{\SetFigFont{9}{10.8}{\rmdefault}{\mddefault}{\updefault}{\color[rgb]{0,0,0}$a_2$}%
}}}}
\put(3701,489){\makebox(0,0)[lb]{\smash{{\SetFigFont{9}{10.8}{\rmdefault}{\mddefault}{\updefault}{\color[rgb]{0,0,0}$a_3$}%
}}}}
\put(3397,-125){\makebox(0,0)[lb]{\smash{{\SetFigFont{9}{10.8}{\rmdefault}{\mddefault}{\updefault}{\color[rgb]{0,0,0}$a_4$}%
}}}}
\put(4388, 16){\makebox(0,0)[lb]{\smash{{\SetFigFont{9}{10.8}{\rmdefault}{\mddefault}{\updefault}{\color[rgb]{0,0,0}$\sqrt{v(q)} = \sqrt{v(p)}$}%
}}}}
\put(1200,-598){\makebox(0,0)[lb]{\smash{{\SetFigFont{9}{10.8}{\rmdefault}{\mddefault}{\updefault}{\color[rgb]{0,0,0}$a_1$}%
}}}}
\put(1594,697){\makebox(0,0)[lb]{\smash{{\SetFigFont{9}{10.8}{\rmdefault}{\mddefault}{\updefault}{\color[rgb]{0,0,0}$a_3$}%
}}}}
\put(1473,-82){\makebox(0,0)[lb]{\smash{{\SetFigFont{9}{10.8}{\rmdefault}{\mddefault}{\updefault}{\color[rgb]{0,0,0}$\sqrt{v(p)}$}%
}}}}
\put(2880, 26){\makebox(0,0)[lb]{\smash{{\SetFigFont{9}{10.8}{\rmdefault}{\mddefault}{\updefault}{\color[rgb]{0,0,0}$\sqrt{u(q)}$}%
}}}}
\end{picture}%
\caption{The quadrilaterals $p$ and $q = F_3 \circ F_4(p)$.}
\label{fig:pq}
\end{figure}
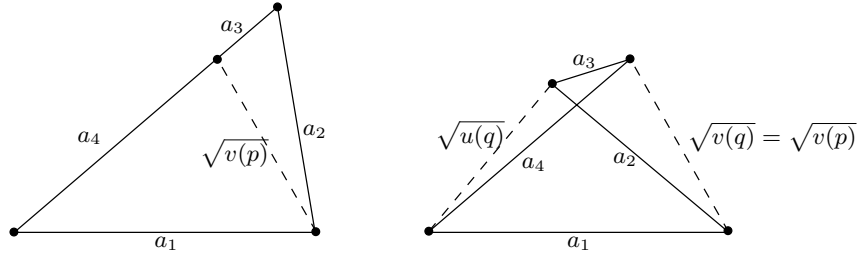

Denote $q = \hat F_3 \circ \hat F_4(p)$. We have $(\hat F_3 \circ \hat F_4)^n = \id$ if and only if $n(p-q) = 0$. The function
\begin{equation}
\label{eqn:x}
x = \frac{u - u(q)}{u - u(p)}
\end{equation}
has a double pole at $p$ and a zero at $q$, and the same branch points as $u$.

Now it remains to compute the values of $x$ at the three remaining branch points:
\[
\frac{(a_1+a_2)^2 - u(q)}{(a_1+a_2)^2 - (a_3 + a_4)^2}, \quad \frac{(a_1-a_2)^2 - u(q)}{(a_1-a_2)^2 - (a_3 + a_4)^2}, \quad \frac{(a_3-a_4)^2 - u(q)}{(a_3-a_4)^2 - (a_3 + a_4)^2}
\]

\begin{lem}
\label{lem:uq}
We have
\begin{gather*}
v(p) = v(q) = \frac{a_3(a_1^2 - a_4^2) + a_4(a_2^2 - a_3^2)}{a_3+a_4}\\
u(q) = \frac{(a_4-a_3)(a_1^2-a_2^2)}{a_3+a_4} + \frac{(a_1^2 - a_2^2 + a_3^2 - a_4^2)(a_1^2a_3^2 - a_2^2a_4^2)}{(a_3+a_4)(a_3(a_1^2 - a_4^2) + a_4(a_2^2 - a_3^2))}
\end{gather*}
\end{lem}
\begin{proof}
By the cosine law we have
\begin{gather*}
v(p) = a_1^2 + a_4^2 + 2a_1a_4 \cos\phi_1(p)\\
a_2^2 = a_1^2 + (a_3+a_4)^2 + 2 a_1(a_3+a_4) \cos\phi_1(p)
\end{gather*}
Expressing $\cos\phi_1(p)$ from the second equation and substituting into the first gives us the formula for $v(p)$.

To compute $u(q)$, rewrite equation \eqref{eqn:DiagCoord} as a quadratic equation in $u$:
\[
v u^2 + (v^2 + 2d_{11} v + d_{10}) u + (d_{01}v + d_{00}) = 0
\]
At $v = v(p)$ this equation has two solutions $u = u(p) = (a_3+a_4)^2$ and $u = u(q)$. Therefore
\[
(a_3+a_4)^2 u(q) = \frac{d_{01} v(p) + d_{00}}{v(p)}
\]
and we arrive at the formula for $u(q)$ stated in the lemma.
\end{proof}

Finally, a lengthy calculation yields the following formulas for the values of $x$ at the branch points:
\[
\frac{(a_1a_3+a_2a_4)^2}{u(p)v(p)}, \quad \frac{(a_1a_3-a_2a_4)^2}{u(p)v(p)}, \quad \frac{(\bar a_1 \bar a_3-\bar a_2 \bar a_4)^2}{u(p)v(p)}
\]
Scaling $x$ by an appropriate functions we obtain a function with the same poles and zeros and with the values
\[
(a_1a_3+a_2a_4)^2, \quad(a_1a_3-a_2a_4)^2, \quad (\bar a_1 \bar a_3-\bar a_2 \bar a_4)^2
\]
at its branch points. This proves Theorem \ref{thm:Period}.

\section{Applications}
\label{sec:App}
\subsection{Ivory's theorem}
\label{sec:Ivory}


The second part of Theorem \ref{thm:NonOr} says that for every quadrilateral with the side lengths $a$ and diagonal lengths $x$ and $y$ there is a quadrilateral with the side lengths $\bar a$ and the same diagonal lengths $x$ and $y$. We call these quadrilaterals \emph{conjugate}.
See Figure \ref{fig:IsomConj} for an example, where we have chosen
\[
a = (10, 5, 6, 3), \quad \bar a = (2, 7, 6, 9)
\]

\begin{figure}
\centering
\begin{picture}(0,0)%
\includegraphics{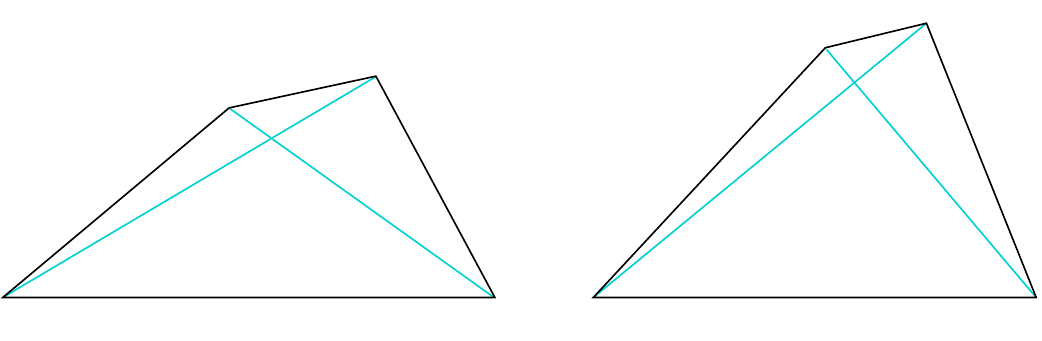}%
\end{picture}%
\setlength{\unitlength}{4144sp}%
\begingroup\makeatletter\ifx\SetFigFont\undefined%
\gdef\SetFigFont#1#2#3#4#5{%
  \reset@font\fontsize{#1}{#2pt}%
  \fontfamily{#3}\fontseries{#4}\fontshape{#5}%
  \selectfont}%
\fi\endgroup%
\begin{picture}(4749,1545)(-11,-1156)
\put(991,-1096){\makebox(0,0)[lb]{\smash{{\SetFigFont{9}{10.8}{\rmdefault}{\mddefault}{\updefault}{\color[rgb]{0,0,0}$a_1$}%
}}}}
\put(1216, 29){\makebox(0,0)[lb]{\smash{{\SetFigFont{9}{10.8}{\rmdefault}{\mddefault}{\updefault}{\color[rgb]{0,0,0}$a_3$}%
}}}}
\put(361,-466){\makebox(0,0)[lb]{\smash{{\SetFigFont{9}{10.8}{\rmdefault}{\mddefault}{\updefault}{\color[rgb]{0,0,0}$a_4$}%
}}}}
\put(3556,-1096){\makebox(0,0)[lb]{\smash{{\SetFigFont{9}{10.8}{\rmdefault}{\mddefault}{\updefault}{\color[rgb]{0,0,0}$\bar a_3$}%
}}}}
\put(3871,254){\makebox(0,0)[lb]{\smash{{\SetFigFont{9}{10.8}{\rmdefault}{\mddefault}{\updefault}{\color[rgb]{0,0,0}$\bar a_1$}%
}}}}
\put(3061,-331){\makebox(0,0)[lb]{\smash{{\SetFigFont{9}{10.8}{\rmdefault}{\mddefault}{\updefault}{\color[rgb]{0,0,0}$\bar a_2$}%
}}}}
\put(811,-556){\makebox(0,0)[lb]{\smash{{\SetFigFont{9}{10.8}{\rmdefault}{\mddefault}{\updefault}{\color[rgb]{0,0,0}$x$}%
}}}}
\put(1396,-556){\makebox(0,0)[lb]{\smash{{\SetFigFont{9}{10.8}{\rmdefault}{\mddefault}{\updefault}{\color[rgb]{0,0,0}$y$}%
}}}}
\put(3331,-601){\makebox(0,0)[lb]{\smash{{\SetFigFont{9}{10.8}{\rmdefault}{\mddefault}{\updefault}{\color[rgb]{0,0,0}$x$}%
}}}}
\put(1981,-421){\makebox(0,0)[lb]{\smash{{\SetFigFont{9}{10.8}{\rmdefault}{\mddefault}{\updefault}{\color[rgb]{0,0,0}$a_2$}%
}}}}
\put(4456,-241){\makebox(0,0)[lb]{\smash{{\SetFigFont{9}{10.8}{\rmdefault}{\mddefault}{\updefault}{\color[rgb]{0,0,0}$\bar a_4$}%
}}}}
\put(4006,-376){\makebox(0,0)[lb]{\smash{{\SetFigFont{9}{10.8}{\rmdefault}{\mddefault}{\updefault}{\color[rgb]{0,0,0}$y$}%
}}}}
\end{picture}%
\caption{A pair of conjugate quadrilaterals.}
\label{fig:IsomConj}
\end{figure}
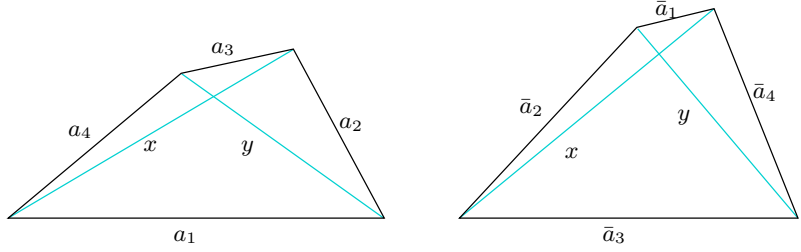

Arseniy Akopyan has noted \cite{Akop12} that the existence of a conjugate to each quadrilateral is equivalent to the Ivory theorem.

\begin{thm}[Ivory's Theorem]
Take two ellipses and two hyperbolas from a confocal family of conics. Then the diagonals of any curvilinear quadrilateral cut out by them have equal lengths.
\end{thm}
\begin{proof}
Consider a quadrilateral with vertices at the foci and at the endpoints of one of the diagonals of the given curvilinear quadrilateral. The order of vertices is chosen so that the foci form a pair of opposite vertices. The quadrilateral formed in the same way but using the second diagonal of the curvilinear quadrilateral has side lengths conjugate to those of the first quadrilateral. Since this pair of conjugate quadrilaterals has one diagonal in common (the segment between the foci), the other diagonals have equal lengths. See Figure \ref{fig:Ivory}.
\end{proof}

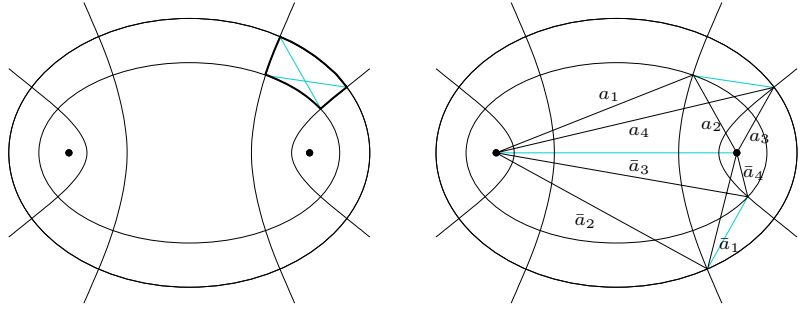
\begin{figure}[ht]
\centering
\begin{picture}(0,0)%
\includegraphics{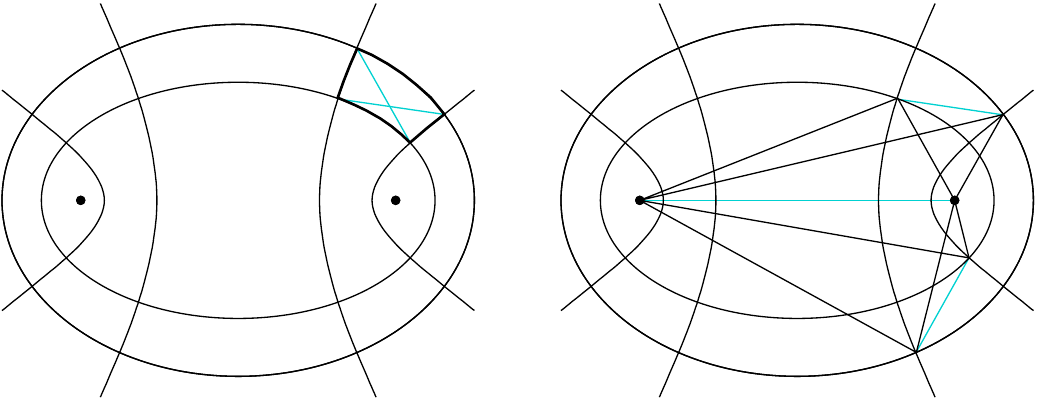}%
\end{picture}%
\setlength{\unitlength}{3315sp}%
\begingroup\makeatletter\ifx\SetFigFont\undefined%
\gdef\SetFigFont#1#2#3#4#5{%
  \reset@font\fontsize{#1}{#2pt}%
  \fontfamily{#3}\fontseries{#4}\fontshape{#5}%
  \selectfont}%
\fi\endgroup%
\begin{picture}(5919,2274)(-4556,-298)
\put(991,929){\makebox(0,0)[lb]{\smash{{\SetFigFont{7}{8.4}{\rmdefault}{\mddefault}{\updefault}{\color[rgb]{0,0,0}$a_3$}%
}}}}
\put( 91,704){\makebox(0,0)[lb]{\smash{{\SetFigFont{7}{8.4}{\rmdefault}{\mddefault}{\updefault}{\color[rgb]{0,0,0}$\bar a_3$}%
}}}}
\put(-134,1244){\makebox(0,0)[lb]{\smash{{\SetFigFont{7}{8.4}{\rmdefault}{\mddefault}{\updefault}{\color[rgb]{0,0,0}$a_1$}%
}}}}
\put( 91,974){\makebox(0,0)[lb]{\smash{{\SetFigFont{7}{8.4}{\rmdefault}{\mddefault}{\updefault}{\color[rgb]{0,0,0}$a_4$}%
}}}}
\put(-314,299){\makebox(0,0)[lb]{\smash{{\SetFigFont{7}{8.4}{\rmdefault}{\mddefault}{\updefault}{\color[rgb]{0,0,0}$\bar a_2$}%
}}}}
\put(766,119){\makebox(0,0)[lb]{\smash{{\SetFigFont{7}{8.4}{\rmdefault}{\mddefault}{\updefault}{\color[rgb]{0,0,0}$\bar a_1$}%
}}}}
\put(631,1019){\makebox(0,0)[lb]{\smash{{\SetFigFont{7}{8.4}{\rmdefault}{\mddefault}{\updefault}{\color[rgb]{0,0,0}$a_2$}%
}}}}
\put(946,659){\makebox(0,0)[lb]{\smash{{\SetFigFont{7}{8.4}{\rmdefault}{\mddefault}{\updefault}{\color[rgb]{0,0,0}$\bar a_4$}%
}}}}
\end{picture}%
\caption{Ivory's theorem and its proof using conjugate quadrilaterals.}
\label{fig:Ivory}
\end{figure}

With the same argument one can prove the Ivory theorem on the sphere and in the hyperbolic plane, see Section \ref{sec:SphHyp}.


\subsection{Dixon's angle condition in the Burmester mechanism}
The linkage on Figure \ref{fig:Dixon} consists of two small quadrilaterals sharing an edge and of one big quadrilateral having an angle in common with each of the small ones. It is easy to see that a linkage of this form is rigid in general. The following is a necessary condition of its flexibility.

\begin{figure}[ht]
\centering
\begin{picture}(0,0)%
\includegraphics{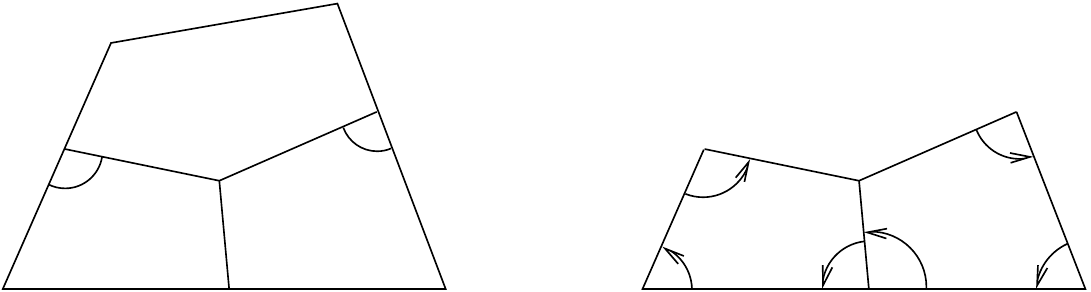}%
\end{picture}%
\setlength{\unitlength}{4144sp}%
\begingroup\makeatletter\ifx\SetFigFont\undefined%
\gdef\SetFigFont#1#2#3#4#5{%
  \reset@font\fontsize{#1}{#2pt}%
  \fontfamily{#3}\fontseries{#4}\fontshape{#5}%
  \selectfont}%
\fi\endgroup%
\begin{picture}(4974,1329)(-11,-523)
\put(3378,-138){\makebox(0,0)[lb]{\smash{{\SetFigFont{9}{10.8}{\rmdefault}{\mddefault}{\updefault}{\color[rgb]{0,0,0}$\theta_1$}%
}}}}
\put(3139,-408){\makebox(0,0)[lb]{\smash{{\SetFigFont{9}{10.8}{\rmdefault}{\mddefault}{\updefault}{\color[rgb]{0,0,0}$\psi_1$}%
}}}}
\put(4204,-357){\makebox(0,0)[lb]{\smash{{\SetFigFont{9}{10.8}{\rmdefault}{\mddefault}{\updefault}{\color[rgb]{0,0,0}$\phi_2$}%
}}}}
\put(3665,-334){\makebox(0,0)[lb]{\smash{{\SetFigFont{9}{10.8}{\rmdefault}{\mddefault}{\updefault}{\color[rgb]{0,0,0}$\phi_1$}%
}}}}
\put(4647,-346){\makebox(0,0)[lb]{\smash{{\SetFigFont{9}{10.8}{\rmdefault}{\mddefault}{\updefault}{\color[rgb]{0,0,0}$\psi_2$}%
}}}}
\put(4331, -2){\makebox(0,0)[lb]{\smash{{\SetFigFont{9}{10.8}{\rmdefault}{\mddefault}{\updefault}{\color[rgb]{0,0,0}$\theta_2$}%
}}}}
\end{picture}%
\caption{An overconstrained, that is generically rigid, linkage.}
\label{fig:Dixon}
\end{figure}
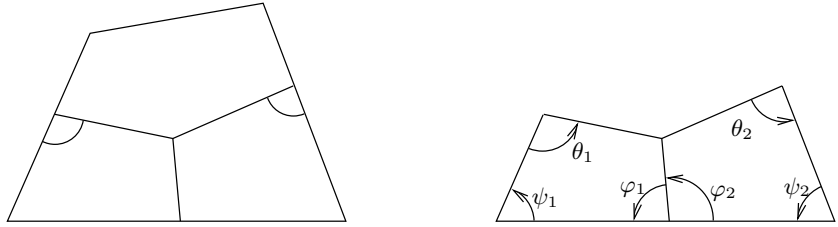

\begin{prp}
If the linkage on Figure \ref{fig:Dixon} is flexible and both small quadrilaterals are of elliptic type, then the angles marked at the left remain either equal modulo $\pi$ or complement each other modulo $\pi$ during the flex:
\[
\theta_1 \pm \theta_2 = n\pi
\]
\end{prp}
\begin{proof}
Forget the big quadrilateral and require only that the horizontal sides of the small quadrilaterals remain aligned. The configuration space $Q_{12}$ of this coupling is a fiber product of the configuration spaces $Q_1$ and $Q_2$ of the small quadrilaterals, see the diagram \eqref{eqn:Z12CD}.

\begin{equation}
\label{eqn:Z12CD}
\xymatrix {
&& Q_{12} \ar[dl] \ar[dr] && \\
& Q_1 \ar[dl] \ar[dr]_{f_1} && Q_2 \ar[dr] \ar[dl]^{f_2} \\
w_1 \in \CP^1 && z \in\CP^1 && \CP^1 \ni w_2
}
\end{equation}

Here $z = z_1 = \tan\frac{\phi_1}2 = \cot\frac{\phi_2}2 = z_2^{-1}$ and $w_i = \tan\frac{\psi_i}2$.

If the branch sets of the maps $f_1$ and $f_2$ on the diagram don't coincide, then $Q_{12}$ is a hyperelliptic curve.
On the other hand, if the coupling remains flexible when restrained by the big quadrilateral, then $Q_{12}$ is the configuration space of the big quadrilateral, thus cannot be hyperelliptic.

Branch points of $f_i$ correspond to $\theta_i \in \{0,\pi\}$. Thus for flexibility it is necessary that
\[
\theta_1 \in \{0,\pi\} \Leftrightarrow \theta_2 \in \{0, \pi\}
\]
Applying Lemma \ref{lem:CosLinDep} to both small quadrilaterals we obtain that $\cos \theta_1$ and $\cos \theta_2$ must be linearly dependent. It follows that either $\cos\theta_1 = \cos \theta_2$, that is $\theta_1 - \theta_2 = n\pi$, or $\cos\theta_1 = -\cos\theta_2$, that is $\theta_1 + \theta_2 = n\pi$.
\end{proof}

It turns out that if the linkage on Figure \ref{fig:Dixon} is flexible, then another bar can be added without restraining its motion, see \cite{Wun68}. This results in a linkage on Figure \ref{fig:BurHart}, left, called the Burmester focal mechanism. Of course, points on the sides of the big quadrilateral may also lie on the extensions of these sides. In particular, if one of them lies at infinity, then the trajectory of the central joint is a line orthogonal to the corresponding side. The result is a straight-line mechanism of Hart \cite{Har77,Kem78,Dar79a}, see Figure \ref{fig:BurHart}, right.

\begin{figure}[ht]
\centering
\includegraphics{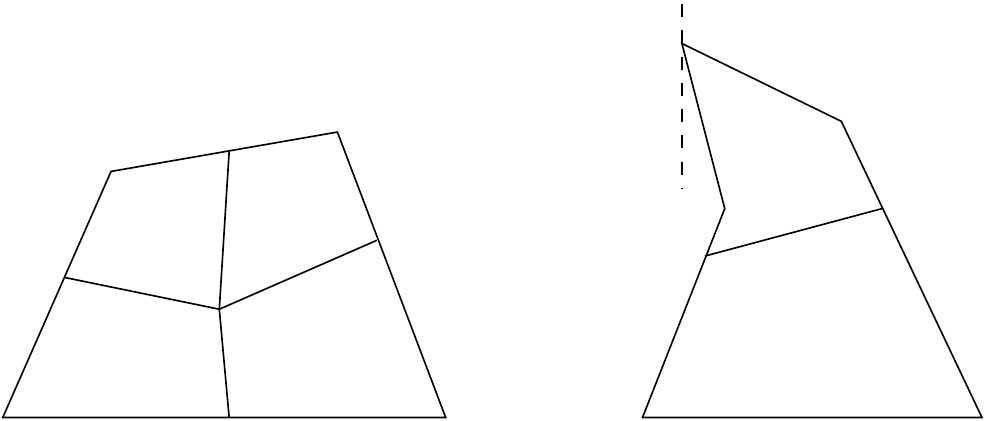}
\caption{Burmester's linkage and its special case, Hart's straight line mechanism.}
\label{fig:BurHart}
\end{figure}

The problem of characterizing all flexible configurations of the form shown on Figure \ref{fig:BurHart}, left, was studied by Krause \cite{Kra08} using Darboux's parametrization of the configuration space. He found three families: in the first case two opposite small quadrilaterals are parallelograms while the other two are similar; in the second all small quadrilaterals are deltoids and the big one is orthodiagonal; in the third family all quadrilaterals are conic or elliptic and the lengths of 12 segments are subject to 7 polynomial relations.

More examples of overconstrained flexible linkages can be found in \cite{Sta13}.

\subsection{Flexible polyhedra}
\label{sec:FlexPol}
Consider a polyhedral cone made of four faces meeting along four edges. If the edges are viewed as ideal hinges, and the faces are allowed to cross each other, then the configuration space of the cone is that of a spherical quadrilateral obtained by intersecting the cone with a sphere centered at the apex. This allows to study the flexibility of polyhedral surfaces with vertices of degree $4$ by means of elliptic functions.

The only closed polyhedron with all vertices of degree $4$ is the octahedron. Bricard \cite{Bri97} studied the system of biquadratic equations between the tangents of dihedral half-angles and described three classes of flexible (self-intersecting) octahedra. In fact, it suffices to look at the three dihedral angles adjacent to the same face, so Bricard's method was to find when the resultant of two biquadratic polynomials $P_1(z_2, z_3)$ and $P_2(z_1, z_3)$ is divisible by a biquadratic polynomial $P_3(z_1,z_2)$.

Recently, Gaifullin \cite{Gai13} classified flexible cross-polytopes (analogs of the octahedron in higher dimensions) in the spaces of constant curvature. Unexpectedly enough (polyhedra of higher dimensions tend to be ``more rigid''), flexible cross-polytopes exist in all dimensions, besides, in $\Sph^3$ there is a type that is not present in Bricard's classification for $\R^3$. Gaifullin used parametrization of dihedral angles by Jacobi elliptic functions, similar to the one in Section \ref{sec:EllReal}.

\begin{figure}[ht]
\centering
\includegraphics[width=.8\textwidth]{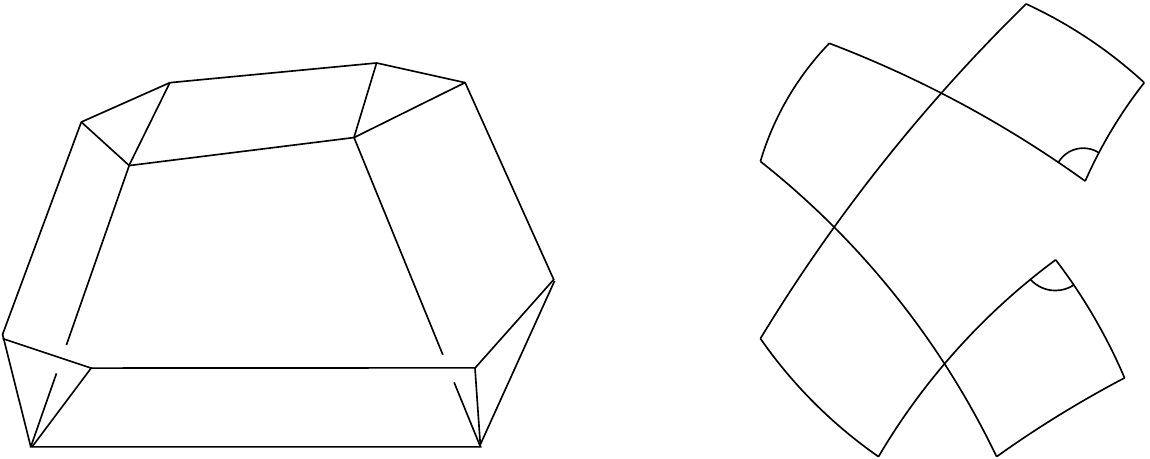}
\caption{A Kokotsakis polyhedron and the associated spherical linkage.}
\label{fig:Koko4}
\end{figure}

There are plenty of non-closed $3$-dimensional polyhedra with all vertices of degree $4$. The neighborhood of a face of such a polyhedron consists of an $n$-gon surrounded by a ``belt'' of triangles and quadrilaterals, and is called a \emph{Kokotsakis polyhedron}, \cite{Kok33}. See Figure \ref{fig:Koko4}, left, for a polyhedron with a quadrangular base. All polygons are required to remain flat during a deformation. Thus the polyhedron is flexible if and only if the spherical framework with scissors-like joints on Figure \ref{fig:Koko4}, right, deforms in such a way that the two marked angles remain equal.

The configuration space of a Kokotsakis polyhedron is described by a system of $n$ biquadratic equations in $n$ variables. Kokotsakis \cite{Kok33}, Graf and Sauer \cite{SG31}, and Stachel \cite{Sta10} described several types of flexible polyhedra. In \cite{Izm_Koko} a complete classification in the case of a quadrangular base was obtained. Parametrization by elliptic functions is used in an essential way; among the new flexible types there is an irreducible one, where the resultant of $P(z_1,z_2)$ and $Q(z_2,z_3)$ is an irreducible polynomial and thus proportional to the resultant of $R(z_3,z_4)$ and $S(z_4,z_1)$.

\section{Configuration spaces of spherical and hyperbolic quadrilaterals}
\label{sec:SphHyp}
Qualitatively, all results of this article apply for quadrilaterals on the sphere or on the hyperbolic plane. Configuration spaces of spherical quadrilaterals play a role in the study of flexible polyhedra, see Section \ref{sec:FlexPol}.

For spherical quadrilaterals, the admissible side lengths are subject to inequalities
\begin{subequations}
\begin{equation}
\label{eqn:QuadIneqASph}
0 < a_i < \pi \quad \forall i
\end{equation}
\begin{equation}
\label{eqn:QuadIneqBSph}
a_i < a_j + a_k + a_l < a_i + 2\pi \quad \forall i,
\end{equation}
\end{subequations}
while in the hyperbolic case the conditions are the same \eqref{eqn:QuadIneqA}, \eqref{eqn:QuadIneqB} as in the euclidean. Similarly, in place of equation \eqref{eqn:Grashof} in the spherical case we have
\[
a_1 \pm a_2 \pm a_3 \pm a_4 \equiv 0 (\mod 2\pi)
\]
This adds ``antideltoids'' and ``antiisograms'' to the respective cases in Definition \ref{dfn:Types}.

In terms of the tangents of half-angles the configuration space of oriented quadrilaterals is described by equations \eqref{eqn:OppX} and \eqref{eqn:AdjX}, where in the spherical case we have 
\begin{equation*}
\begin{aligned}
b_{22} &= \sin\frac{a_1 + a_2 - a_3 - a_4}2 \sin\frac{a_1 - a_2 + a_3 - a_4}2\\
b_{20} &= \sin\frac{a_1 + a_2 + a_3 - a_4}2 \sin\frac{a_1 - a_2 - a_3 - a_4}2\\
b_{02} &= \sin\frac{a_1 - a_2 + a_3 + a_4}2 \sin\frac{a_1 + a_2 - a_3 + a_4}2\\
b_{00} &= \sin\frac{a_1 - a_2 - a_3 + a_4}2 \sin\frac{a_1 + a_2 + a_3 + a_4}2
\end{aligned}
\end{equation*}
and similarly for $c_{ij}$ with $c_{11} = -\sin a_2 \sin a_4$. In the hyperbolic case $\sinh$ takes the place of $\sin$.

Further, orthodiagonal spherical quadrilaterals are characterized by
\[
\cos a_1 \cos a_3 = \cos a_2 \cos a_4
\]
and these are the only $2$-periodic ones.

There are analogs of identities from Lemmas \ref{lem:SIdent} and \ref{lem:SIdent2}. Some of them are collected in the next lemma.

\begin{lem}
\label{lem:SIdentSph}
For any $a, b, c, d$ the following identities hold.
\begin{gather*}
\cos a \cos b + \cos c \cos d = \cos \bar a \cos \bar b + \cos \bar c \cos \bar d\\
\sin a \sin b + \sin c \sin d = \sin \bar a \sin \bar b + \sin \bar c \sin \bar d\\
\cos a \cos b - \cos c \cos d = \sin \bar a \sin \bar b - \sin \bar c \sin \bar d\\
\sin a \sin b \sin c \sin d + \cos a \cos b \cos c \cos d = \sin \bar a \sin \bar b \sin \bar c \sin \bar d + \cos \bar a \cos \bar b \cos \bar c \cos \bar d
\end{gather*}
Here $\bar a = \frac12(-a + b + c + d)$ etc.
The same identities hold when $\sin$ and $\cos$ are replaced by $\sinh$ and $\cosh$.
\end{lem}

Parametrization of the configuration space in terms of Jacobi elliptic functions (or trigonometric functions, for conic quadrilaterals) are derived in the spherical case in \cite{Izm_Koko}.

Equation in terms of diagonal lengths, describing the space of quadrilaterals up to possibly orientation changing congruences undergoes a more significant change.
\begin{lem}
\label{lem:DiagCoordSph}
Let $x$ and $y$ be the lengths of diagonals in a spherical quadrilateral with side lengths $a_1, a_2, a_3, a_4$ (in this cyclic order). Denote
\[
u = \cos x, \quad v = \cos y
\]
Then the configuration space of quadrilaterals up to congruence is an algebraic curve given by the equation
\begin{equation*}
\label{eqn:DiagCoordSph}
u^2 v^2 - u^2 - v^2 + 2d_{11} uv + d_{10}x + d_{01}y + d_{00} = 0, \quad \text{where}
\end{equation*}
\begin{equation*}
\begin{aligned}
d_{11} &= - (\cos a_1 \cos a_3 + \cos a_2 \cos a_4)\\
d_{10} &= 2(\cos a_1 \cos a_2 + \cos a_3 \cos a_4)\\
d_{01} &= 2(\cos a_1 \cos a_4 + \cos a_2 \cos a_3)\\
d_{00} &= 1 - \cos^2 a_1 - \cos^2 a_2 - \cos^2 a_3 - \cos^2 a_4 + (\cos a_1 \cos a_3 - \cos a_2 \cos a_4)^2
\end{aligned}
\end{equation*}

The same holds for the configuration space of hyperbolic quadrilaterals, with $\cos$ and $\sin$ replaced by $\cosh$ and $\sinh$.
\end{lem}
\begin{proof}
Consider the Gram matrix of the vectors from the center of the sphere (or center of the hyperboloid, in the hyperbolic case) to the vertices of the quadrilateral. Its determinant must be equal to zero.
\end{proof}

In the hyperbolic case, a similar equation is derived by considering the Minkowski space model of the hyperbolic plane.

\begin{thm}[Ivory's theorem on the sphere and in the hyperbolic plane]
\label{thm:IvorySph}
Let $A_1$ and $A_2$ be two points on the sphere, respectively in the hyperbolic plane.
Define an ellipse with the foci $A_1$ and $A_2$ as the set of points with a constant sum of geodesic distances to $A_1$ and $A_2$, and a hyperbola as the set of points with a constant absolute value of the difference of these distances.

Then in every curvilinear quadrilateral cut out by a confocal family of ellipses and hyperbolas the diagonals have equal geodesic lengths.
\end{thm}
Note that a spherical hyperbola consists of two spherical ellipses (replace one of the foci by its antipode).
\begin{proof}
Coefficients of the equation \eqref{eqn:DiagCoordSph} don't change when $a_i$ are replaced by $\bar a_i$. For the coefficients $d_{11}$, $d_{10}$, $d_{01}$ this is immediate from Lemma \ref{lem:SIdentSph}, while $d_{00}$ must first be shown to be equal to
\begin{multline*}
-1 + (\sin a_1 \sin a_3 + \sin a_2 \sin a_4)^2\\
- 2 \sin a_1 \sin a_2 \sin a_3 \sin a_4 - 2 \cos a_1 \cos a_2 \cos a_3 \cos a_4
\end{multline*}
(with hyperbolic trigonometric functions in the hyperbolic case).
Now the Ivory's theorem follows by the same argument as in the euclidean case, see Section \ref{sec:Ivory}.
\end{proof}

It is not clear to us whether Theorem \ref{thm:IvorySph} is related to a version of Ivory's theorem for hyperbolic spaces presented in \cite{SW04}.

Finally, note that while in the euclidean case the scaling of the sides
\[
(a_1, a_2, a_3, a_4) \mapsto (\lambda a_1, \lambda a_2, \lambda a_3, \lambda a_4)
\]
doesn't change the configuration space, in the spherical and hyperbolic cases it does.

\section{Open questions}
\begin{prb}
The area $S$ is a meromorphic function on the space of quadrilaterals with fixed side lengths. The inscribed quadrilateral $I(a)$ is a branch point of $A$, and we have $S(I(a)) = \sqrt{\bar a_1 \bar a_2 \bar a_3 \bar a_4}$ (Brahmagupta's formula). What are the other branch points and other critical values of $A$? Is there a nice theory in the spherical and hyperbolic case?
\end{prb}

\begin{prb}
Derive periodicity conditions for foldings of spherical and hyperbolic quadrilaterals.
\end{prb}

\begin{prb}
Study flexible bipyramids (or suspensions) with the help of parametrizations of the configuration spaces of spherical quadrilaterals.
\end{prb}
A bipyramid over a quadrangle is an octahedron, thus this is a generalization of Bricard's problem of flexible octahedra. Flexible bipyramids were studied in \cite{Con74,AC11}. It is not known whether there exist non-trivial flexible bipyramids over the pentagon; in the articles cited examples of flexible bipyramids over the hexagon were found. By cutting off the two apices of a bipyramid one obtains a polyhedron with vertices of degree four. This leads to a system of biquadratic equations, see Section \ref{sec:FlexPol}. Besides this system, there is also a holonomy condition (since the polyhedron is not simply connected).

\begin{prb}
The dihedral group action on $a = (a_1, a_2, a_3, a_4)$ does not change the configuration space $\QorC(a)$. What happens when the cyclic order of sides is changed in one of the following two ways?
\[
a' = (a_1, a_3, a_2, a_4), \quad a'' = (a_1, a_2, a_4, a_3)
\]
\end{prb}

\begin{figure}[ht]
\centering
\begin{picture}(0,0)%
\includegraphics{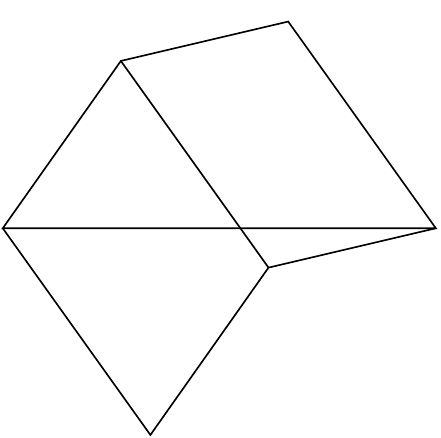}%
\end{picture}%
\setlength{\unitlength}{4144sp}%
\begingroup\makeatletter\ifx\SetFigFont\undefined%
\gdef\SetFigFont#1#2#3#4#5{%
  \reset@font\fontsize{#1}{#2pt}%
  \fontfamily{#3}\fontseries{#4}\fontshape{#5}%
  \selectfont}%
\fi\endgroup%
\begin{picture}(2004,1992)(-11,-1468)
\put(451,-466){\makebox(0,0)[lb]{\smash{{\SetFigFont{9}{10.8}{\rmdefault}{\mddefault}{\updefault}{\color[rgb]{0,0,0}$a_1$}%
}}}}
\put(181,-1096){\makebox(0,0)[lb]{\smash{{\SetFigFont{9}{10.8}{\rmdefault}{\mddefault}{\updefault}{\color[rgb]{0,0,0}$a_2$}%
}}}}
\put( 91,-61){\makebox(0,0)[lb]{\smash{{\SetFigFont{9}{10.8}{\rmdefault}{\mddefault}{\updefault}{\color[rgb]{0,0,0}$a_4$}%
}}}}
\put(766,389){\makebox(0,0)[lb]{\smash{{\SetFigFont{9}{10.8}{\rmdefault}{\mddefault}{\updefault}{\color[rgb]{0,0,0}$a_3$}%
}}}}
\put(1666,-16){\makebox(0,0)[lb]{\smash{{\SetFigFont{9}{10.8}{\rmdefault}{\mddefault}{\updefault}{\color[rgb]{0,0,0}$a_2$}%
}}}}
\put(901,-196){\makebox(0,0)[lb]{\smash{{\SetFigFont{9}{10.8}{\rmdefault}{\mddefault}{\updefault}{\color[rgb]{0,0,0}$a_2$}%
}}}}
\put(766,-1096){\makebox(0,0)[lb]{\smash{{\SetFigFont{9}{10.8}{\rmdefault}{\mddefault}{\updefault}{\color[rgb]{0,0,0}$a_4$}%
}}}}
\put(1441,-736){\makebox(0,0)[lb]{\smash{{\SetFigFont{9}{10.8}{\rmdefault}{\mddefault}{\updefault}{\color[rgb]{0,0,0}$a_3$}%
}}}}
\end{picture}%
\caption{Linking the configuration spaces $\QorC(a)$, $\QorC(a')$, and $\QorC(a'')$.}
\label{fig:Permute}
\end{figure}

The curves $\QorC(a)$, $\QorC(a')$, and $\QorC(a'')$ are isomorphic, as follows from Theorem \ref{thm:Param} or as is indicated on Figure \ref{fig:Permute}. What does change, are the parametrizations. As formulas in Proposition \ref{prp:SnEuc} suggest, the shifts start mixing up with the amplitudes. The configuration space of the linkage on Figure \ref{fig:Permute} seems to be a four-fold covering of each of the spaces  $\QorC(a)$, $\QorC(a')$, and $\QorC(a'')$.

\end{document}